\documentclass[12pt,a4paper,reqno]{amsart}
\usepackage[a4paper,left=30mm,right=30mm,top=36mm,bottom=36mm]{geometry}

\usepackage{color}
\usepackage{amssymb}
\usepackage{amsmath}
\usepackage{graphicx}
\usepackage{float}
\usepackage[utf8]{inputenc}
\usepackage[english]{babel}
\usepackage{stmaryrd}
\usepackage{mathtools}
\usepackage{subcaption}
\usepackage{pgfplots}
\usepackage{grffile}
\pgfplotsset{compat=1.18}

\newcommand{\R}{\mathbb{R}}

\newcommand{\N}{\mathbb{N}}

\newcommand{\eps}{\varepsilon}
\newcommand{\fhi}{\varphi}
\newcommand{\weak}{\rightharpoonup}

\newcommand{\vertiii}[1]{{\left\vert\kern-0.25ex\left\vert\kern-0.25ex\left\vert #1
    \right\vert\kern-0.25ex\right\vert\kern-0.25ex\right\vert}}

\def\calA{\mathcal{A}}
\def\calB{\mathcal{B}}

\def\calO{\mathcal{O}}

\def\calT{\mathcal{T}}

\def\XXint#1#2#3{{\setbox0=\hbox{$#1{#2#3}{\int}$}
     \vcenter{\hbox{$#2#3$}}\kern-.5\wd0}}

\numberwithin{equation}{section}

\usepackage[]{hyperref}
\usepackage[noabbrev, capitalise, nameinlink]{cleveref}
\crefname{equation}{}{}
\crefname{assumption}{Assumption}{Assumptions}
\crefname{rmrk}{Remark}{Remarks}
\crefformat{equation}{\textup{#2(#1)#3}}

\usepackage[colorinlistoftodos,prependcaption]{todonotes}

\newtheorem{remark}{Remark}[section]    
\newtheorem{theorem}{Theorem}[section]    
   
\newtheorem{lemma}{Lemma}[section]	
\newtheorem{corollary}{Corollary}[section]  
\newtheorem{proposition}{Proposition}[section]

\begin{document}

\title[FEM for stationary FPK equations and numerical homogenization]{Finite element approximation of stationary Fokker--Planck--Kolmogorov equations with application to periodic numerical homogenization}

\author[T. Sprekeler]{Timo Sprekeler}
\address[T. Sprekeler]{Department of Mathematics, Texas A{\&}M University, College Station, TX 77843, USA.}
\email{timo.sprekeler@tamu.edu}

\author[E. S\"{u}li]{Endre S\"{u}li}
\address[E. S\"{u}li]{University of Oxford, Mathematical Institute, Woodstock Road, Oxford OX2 6GG, UK.}
\email{suli@maths.ox.ac.uk}

\author[Z. Zhang]{Zhiwen Zhang}
\address[Z. Zhang]{The University of Hong Kong, Department of Mathematics, Pokfulam Road, Hong Kong, China.}
\email{zhangzw@hku.hk}

\date{\today}

\begin{abstract}

We propose and rigorously analyze a finite element method for the approximation of stationary Fokker--Planck--Kolmogorov (FPK) equations subject to periodic boundary conditions in two settings: one with weakly differentiable coefficients, and one with merely essentially bounded measurable coefficients under a Cordes-type condition. These problems arise as governing equations for the invariant measure in the homogenization of nondivergence-form equations with large drifts. In particular, the Cordes setting guarantees the existence and uniqueness of a square-integrable invariant measure. We then suggest and rigorously analyze an approximation scheme for the effective diffusion matrix in both settings, based on the finite element scheme for stationary FPK problems developed in the first part. Finally, we demonstrate the performance of the methods through numerical experiments.

\end{abstract}

\maketitle

\section{Introduction}

In the first part of this paper, we consider the numerical approximation of the stationary Fokker--Planck--Kolmogorov-type problem
\begin{align}\label{FPK intro}
\begin{split}
-D^2:(Au) + \nabla\cdot (bu) := -\sum_{i,j=1}^n \partial_{ij}^2 (a_{ij}u) + \sum_{k=1}^n \partial_k(b_k u) = \nabla\cdot F \quad\text{in } Y,&\\ u\text{ is $Y$-periodic},&
\end{split}
\end{align}
where $Y:=(0,1)^n\subset \R^n$ denotes the unit cell, $F\in L^2_{\mathrm{per}}(Y;\R^n)$, and $(A,b)$ is a pair of given coefficients 
\begin{align*}
A = (a_{ij})_{1\leq i,j\leq n}\in L^{\infty}_{\mathrm{per}}(Y;\R^{n\times n}_{\mathrm{sym}}),\qquad b = (b_k)_{1\leq k\leq n}\in L^{\infty}_{\mathrm{per}}(Y;\R^n),
\end{align*}
where $A$ is assumed to be uniformly elliptic (see \eqref{uniform ell}). In addition, we make one of the following two assumptions:
\begin{itemize}
\item Setting $\calA$: we assume that $A\in W^{1,p}_{\mathrm{per}}(Y;\R^{n\times n}_{\mathrm{sym}})$ for some $p>n$.
\item Setting $\calB$: we assume that the coefficients satisfy the Cordes-type condition
\begin{align}\label{mod Cordes intro}
\exists\, \delta\in \left(\frac{n}{n+\pi^2},1\right]:\qquad\frac{\lvert A\rvert^2 + \lvert b\rvert^2}{(\mathrm{tr}(A))^2} \leq \frac{1}{n-1+\delta}\quad\text{a.e. in }\overline{Y},
\end{align}
which can be relaxed to the classical Cordes condition \eqref{classical Cordes intro} if $\lvert b \rvert = 0$ a.e..
\end{itemize}
Here, we used the notation $\lvert A\rvert^2 := A:A$ and $\lvert b\rvert^2 :=b\cdot b$. We call the condition \eqref{mod Cordes intro} a Cordes-type condition as it is inspired by the classical Cordes condition
\begin{align}\label{classical Cordes intro}
\exists\, \delta\in \left(0,1\right]:\qquad\frac{\lvert A\rvert^2}{(\mathrm{tr}(A))^2} \leq \frac{1}{n-1+\delta}\quad\text{a.e.}
\end{align}
used in the study of nondivergence-form equations $-A:D^2 v = f$, and the Cordes-type condition
\begin{align}\label{Cordes Abc intro}
\exists\, (\delta,\lambda)\in \left(0,1\right]\times (0,\infty):\qquad\frac{\lvert A\rvert^2 + \frac{1}{2\lambda}\lvert b\rvert^2 + \frac{1}{\lambda^2}c^2}{(\mathrm{tr}(A)+\frac{1}{\lambda}c)^2} \leq \frac{1}{n+\delta}\quad\text{a.e.}
\end{align} 
used in the study of nondivergence-form equations $-A:D^2 v - b\cdot \nabla v + c v = f$; see \cite{Cor56,SS13,SS14}. It is worth noting that \eqref{Cordes Abc intro} can never be satisfied when $c = 0$ since $\frac{\lvert M\rvert^2}{(\mathrm{tr}(M))^2} \geq \frac{1}{n}$ for any $M\in \R^{n\times n}$. 

Problems of the form \eqref{FPK intro} arise naturally as the governing equation for the invariant measure, i.e., the solution to
\begin{align}\label{r problem intro}
-D^2:(Ar) + \nabla\cdot (br) = 0\quad\text{in } Y,\qquad  r\text{ is $Y$-periodic},\qquad \int_Y r = 1,
\end{align}
which is used to determine the effective problem in the periodic homogenization of nondivergence-form equations with large drifts, i.e.,
\begin{align}\label{ueps problem}
\begin{split}
-A\left(\frac{\cdot}{\eps}\right) :D^2 u_\eps  - \frac{1}{\eps}\,b\left(\frac{\cdot}{\eps}\right)\cdot \nabla u_\eps &= f\quad\text{in }\Omega,\\  u_{\eps} &= g\quad\text{on }\partial\Omega,
\end{split}
\end{align}
which is the focus of the second part of this paper. In setting $\calA$, it is known that \eqref{r problem intro} has a unique positive H\"{o}lder continuous solution and that if the drift satisfies the centering condition $\int_Y rb = 0$ and $f,g,\Omega$ are sufficiently regular, then the sequence of solutions $(u_{\eps})_{\eps>0}$ to \eqref{ueps problem} converges weakly in $H^1(\Omega)$ as $\eps\rightarrow 0$ to the solution $\overline{u}$ of the effective problem
\begin{align*}
-\overline{A}:D^2 \overline{u} &= f\quad\text{in }\Omega,\\  \overline{u} &= g\quad\text{on }\partial\Omega,
\end{align*}
where the effective diffusion matrix $\overline{A}\in \R^{n\times n}$ is the symmetric positive definite matrix given by
\begin{align}\label{def Abar intro}
\overline{A} := \int_Y r[I_n+D\chi]A[I_n + (D\chi)^{\mathrm{T}}]
\end{align}
with $\chi := (\chi_1,\dots,\chi_n)$ and $\chi_j$ denoting the solution to 
\begin{align}\label{chij pro intro}
-A:D^2 \chi_j - b\cdot\nabla \chi_j = b_j\quad\text{in } Y,\qquad \chi_j\text{ is $Y$-periodic},\qquad \int_{Y} \chi_j = 0
\end{align}
for $1\leq j\leq n$; see, e.g., \cite{BLP11,ES08,JKO94}. The goal of the second part of this paper is the efficient and accurate numerical approximation of the effective diffusion matrix $\overline{A}$ from \eqref{def Abar intro} for both settings $\calA$ and $\calB$. 

In setting $\calA$, we approximate the solutions to \eqref{r problem intro} and \eqref{FPK intro} by rewriting the problems in divergence-form and adapting Schatz's method, see \cite{Sch74}, to the periodic setting to handle the resulting noncoercive variational form. Regarding the numerical approximation of \eqref{chij pro intro}, we multiply the equation by the approximation of the invariant measure $r$ and only then rewrite the problem in divergence-form in order to overcome the low regularity of solutions to the dual problem to \eqref{chij pro intro}. Our approximation of the effective diffusion matrix \eqref{def Abar intro} relies on a combination of the finite element schemes for $r$ and $\chi_j$.

In setting $\calB$, we show that \eqref{r problem intro} has a unique nonnegative solution $r\in L^2_{\mathrm{per}}(Y)$, and we suggest and rigorously analyze a finite element method for its approximation which is based on the observation that $r$ is of the form
\begin{align*}
r = C\frac{\mathrm{tr}(A)}{\lvert A\rvert^2 + \lvert b\rvert^2} (1-\nabla\cdot \rho),
\end{align*}  
where $C>0$ is a constant and $\rho$ is the unique solution of a Lax--Milgram-type problem in $H^1_{\mathrm{per},0}(Y;\R^n)$, i.e., the subspace of $H^1_{\mathrm{per}}(Y;\R^n)$ consisting of functions with mean zero. Further, we suggest and rigorously analyze a finite element method for the approximation of solutions to \eqref{FPK intro}, which is based on the observation that any solution $u\in L^2_{\mathrm{per}}(Y)$ to \eqref{FPK intro} is of the form
\begin{align*}
u = \frac{\mathrm{tr}(A)}{\lvert A\rvert^2 + \lvert b\rvert^2}(-\nabla\cdot \tilde{\rho}_0)+cr,
\end{align*}
where $c\in \R$ is a constant and $\tilde{\rho}_0$ is the unique solution of a Lax--Milgram-type problem in $H^1_{\mathrm{per},0}(Y;\R^n)$. Regarding the problem \eqref{chij pro intro}, we show that under the centering condition $\int_Y rb = 0$ there exists a unique solution $\chi_j\in H^2_{\mathrm{per}}(Y)$ and we suggest and analyze a finite element method for its approximation. Our approximation of the effective diffusion matrix \eqref{def Abar intro} relies on a combination of the finite element schemes for $r$ and $\chi_j$. We also discuss assumptions under which homogenization occurs in this setting.

The construction of finite element methods for FPK-type equations has received considerable attention over the past decades; we refer to \cite{BWJ96,BS68,KN06,Lan85,Lan91,MB05} for some of the early developments in the case of smooth coefficients. The distinguishing feature of our study in Setting $\calA$ is the rigorous error analysis for the periodic setting. For the case of merely essentially bounded measurable coefficients, there are very few publications, including the recent primal-dual weak Galerkin approach from \cite{LW23,WW20}. In our study for Setting $\calB$, we choose a different route by imposing a new Cordes-type condition, performing a suitable renormalization of the problem, and developing a simple finite element framework inspired by the prior works \cite{Gal17b,SS13,Spr24} on nondivergence-form problems.

Regarding the homogenization of linear elliptic equations in nondivergence-form, we refer to \cite{CSS20,GST22,GTY20,KL16,Spr24,ST21} for recent developments in periodic homogenization, to \cite{AL89,JZ23} for essential transformation procedures, and to \cite{AFL22,AL17,AS14,GT22,GT23} for recent developments in stochastic homogenization.

In recent years, significant progress has been made on the numerical homogenization of equations in nondivergence-form; see, e.g., \cite{CSS20,FGP24,Spr24} for linear equations and \cite{GSS21,KS22,QST24} for Hamilton--Jacobi--Bellman equations. Concerning the numerical homogenization of divergence-form problems with large drifts, we refer to \cite{BLL24,BFP24,HO10,LLM17,LPS18,ZC23} and the references therein. To be best of our knowledge, we are not aware of any previous research on developing finite element methods for approximating effective diffusion matrices in the context of nondivergence-form equations with large drift terms.

Finally, we refer to \cite{FP94,Gar97,HP08,LOY98,PS08} and the references therein for further contributions to the homogenization of convection-diffusion equations via probabilistic methods, and to \cite{WXZ21} for a stochastic structure-preserving scheme for computing the effective diffusivity of three-dimensional periodic or chaotic flows.

We briefly explain the organization of the paper. In Section \ref{Sec: Hom FPK}, we study the finite element approximation of stationary FPK equations subject to periodic boundary conditions, where Setting $\calA$ is discussed in Section \ref{Sec: 2A} and Setting $\calB$ is discussed in Section \ref{Sec: setting B}, respectively. In Section \ref{Sec: NonhomFPK}, we extend our results to stationary FPK-type problems of the form \eqref{FPK intro}. After that, we study the finite element approximation of the nondivergence-form problem \eqref{chij pro intro} and the numerical approximation of the effective diffusion matrix \eqref{def Abar intro} in Section \ref{Sec: SecAbar}, where Section \ref{Sec: SecAbar A} focuses on Setting $\calA$ and Section \ref{Sec: SecAbar B} focuses on Setting $\calB$. Finally, we demonstrate the theoretical results in numerical experiments provided in Section \ref{Sec: NumExp}.

\section{FEM for Stationary FPK Problems}\label{Sec: Hom FPK}

In this section, we discuss the finite element approximation of stationary FPK equations subject to periodic boundary conditions, i.e.,
\begin{align}\label{q problem}
-D^2:(Au) + \nabla\cdot (bu) = 0\quad\text{in } Y,\qquad u\text{ is $Y$-periodic},
\end{align} 
where $Y:=(0,1)^n$ denotes the unit cell in $\R^n$. We will always assume that
\begin{align}\label{Linfty}
A\in L^\infty_{\mathrm{per}}(Y;\R^{n\times n}_{\mathrm{sym}}),\qquad b\in L^\infty_{\mathrm{per}}(Y;\R^n),
\end{align}
and that $A$ is uniformly elliptic, i.e., 
\begin{align}\label{uniform ell}
\exists \lambda,\Lambda>0:\qquad \lambda I_n\leq A\leq \Lambda I_n\text{ a.e. in }\R^n.
\end{align}
We will consider two settings -- one with an additional regularity assumption on the coefficients, and one with a Cordes-type condition on the coefficients but without any additional regularity assumptions.
\begin{itemize}
\item Setting $\calA$ (higher regularity): We write $(A,b)\in \calA$ if \eqref{Linfty}, \eqref{uniform ell} hold and
\begin{align}\label{AinW1p}
A\in W^{1,p}_{\mathrm{per}}(Y;\R^{n\times n}_{\mathrm{sym}})\text{ for some }p>n,
\end{align}
where we always assume that $p\geq 2$.
\item Setting $\calB$ (Cordes-type): We write $(A,b)\in \calB$ if \eqref{Linfty}, \eqref{uniform ell} hold and
\begin{align}\label{mod Cordes}
\exists\, \delta\in \left(\frac{n}{n+\pi^2},1\right]:\qquad\frac{\lvert A\rvert^2 + \lvert b\rvert^2}{(\mathrm{tr}(A))^2} \leq \frac{1}{n-1+\delta}\quad\text{a.e. in }\R^n,
\end{align}
where $\lvert A\rvert := \sqrt{A:A}$ denotes the Frobenius norm of $A$.
\end{itemize}
The condition \eqref{mod Cordes} resembles the Cordes condition; see, e.g., \cite{Cor56,SS13,Tal65}. 

Throughout the paper, we use the notation $L^2_{\mathrm{per},0}(Y):=\{v\in L^2_{\mathrm{per}}(Y):\int_Y v = 0\}$ and $H^k_{\mathrm{per},0}(Y):=\{v\in H^k_{\mathrm{per}}(Y):\int_Y v = 0\}$ for $k\in \N$.

\subsection{Setting $\calA$}\label{Sec: 2A}

In this section, we study the well-posedness and the finite element approximation of solutions to the FPK problem \eqref{q problem} for the case $(A,b)\in\calA$.

\subsubsection{Well-posedness} 

First, we note that when $(A,b)\in\calA$, we can rewrite the problem \eqref{q problem} in divergence-form thanks to \eqref{AinW1p}, i.e.,
\begin{align*}
-\nabla\cdot\left(A\nabla u + (\mathrm{div}(A)-b) u\right) = 0\quad\text{in } Y,\qquad u\text{ is $Y$-periodic}.
\end{align*}
Let us also briefly note that by \eqref{AinW1p} and Sobolev embeddings, we have that $A\in C^{0,\alpha}(\R^n;\R^{n\times n}_{\mathrm{sym}})$ for some $\alpha > 0$. Then, \eqref{q problem} has a unique H\"{o}lder continuous solution up to multiplicative constants. More precisely, the following result is known to hold; see, e.g., \cite{BLP11,BKR01}.
\begin{proposition}[Analysis of \eqref{q problem} in setting $\calA$]\label{Prop: r}
Let $(A,b)\in \calA$. Then, there exists a unique solution $r\in H^1_{\mathrm{per}}(Y)$ to the problem 
\begin{align}\label{r problem}
-D^2:(Ar) + \nabla\cdot (br) = 0\quad\text{in } Y,\qquad r\text{ is $Y$-periodic},\qquad \int_{Y} r = 1,
\end{align} 
and any solution $u\in H^1_{\mathrm{per}}(Y)$ to \eqref{q problem} is a constant multiple of $r$. Further, $r\in W^{1,p}_{\mathrm{per}}(Y)$ with $p>n$ as in \eqref{AinW1p}, there holds $\inf_{\R^n}r >0$, and for $g\in L^2_{\mathrm{per}}(Y)$ the problem
\begin{align}\label{ndg}
-A:D^2 v - b\cdot\nabla v = g\quad\text{in } Y,\qquad v\text{ is $Y$-periodic},\qquad \int_{Y} v = 0
\end{align}
admits a (unique) solution $v\in H^1_{\mathrm{per},0}(Y)$ if and only if $\int_Y gr = 0$.
\end{proposition}

Note that if $(A,b)\in \calA$, what we mean by a solution $v\in H^1_{\mathrm{per},0}(Y)$ to \eqref{ndg} is an element $v$ in $H^1_{\mathrm{per},0}(Y)$ that satisfies the natural weak formulation of the problem obtained by rewriting \eqref{ndg} in divergence-form.

\subsubsection{Finite element approximation of \eqref{r problem}}\label{Sec: FEM r Set A}

We now discuss the finite element approximation of \eqref{r problem}. The ideas are inspired by the earlier work \cite{CSS20}. 

We begin by noting that
\begin{align}\label{rhat defn}
\hat{r}:= r - 1 \in H^1_{\mathrm{per},0}(Y)
\end{align}
is the unique solution in $H^1_{\mathrm{per},0}(Y)$ to the divergence-form problem
\begin{align*}
-\nabla\cdot (A\nabla \hat{r}+(\mathrm{div}(A)-b) \hat{r}) = \nabla\cdot (\mathrm{div}(A)-b)\quad\text{in }Y,\;\; \hat{r}\text{ is $Y$-periodic},\;\; \int_{Y} \hat{r} = 0.
\end{align*}
More precisely, $\hat{r}$ is the unique element in $H^1_{\mathrm{per},0}(Y)$ satisfying
\begin{align}\label{rhat weak}
a(\hat{r},v) = \int_Y (b-\mathrm{div}(A))\cdot \nabla v\qquad \forall v\in H^1_{\mathrm{per},0}(Y),
\end{align}
where $a(\cdot,\cdot):H^1_{\mathrm{per},0}(Y)\times H^1_{\mathrm{per},0}(Y) \rightarrow \R$ is given by
\begin{align}\label{a bil form}
a(v_1,v_2):=\int_Y A\nabla v_1\cdot \nabla v_2+\int_Y v_1(\mathrm{div}(A)-b)\cdot \nabla v_2
\end{align}
for $v_1,v_2 \in H^1_{\mathrm{per},0}(Y)$. Clearly, $a$ defines a bounded bilinear form on $H^1_{\mathrm{per},0}(Y)$, i.e.,
\begin{align}\label{a bddness}
\lvert a(v_1,v_2)\rvert \leq C_1 \|v_1\|_{H^1(Y)}\|v_2\|_{H^1(Y)}\qquad \forall v_1,v_2\in H^1_{\mathrm{per},0}(Y) 
\end{align}
for some constant $C_1 > 0$, but the bilinear form $a$ is not coercive, making the finite element approximation of \eqref{rhat weak} nonstandard.

However, since $\mathrm{div}(A)-b\in L^p_{\mathrm{per}}(Y;\R^n)$ with $p>n$, it is easy to show using the assumed uniform ellipticity \eqref{uniform ell}, together with H\"{o}lder, Gagliardo--Nirenberg and Young inequalities, that the following G{\aa}rding  inequality holds:
\begin{align}\label{a Garding}
a(v,v)\geq \frac{\lambda}{2}\|v\|_{H^1(Y)}^2 - C_2\|v\|_{L^2(Y)}^2\qquad \forall v\in H^1_{\mathrm{per},0}(Y)
\end{align}
for some constant $C_2 > 0$. This enables us to prove the following result using an adaptation of Schatz's method \cite{Sch74}. 
\begin{theorem}[Finite element approximation of \eqref{r problem}]\label{Thm: Approx r in set A}
Let $(A,b)\in \calA$. Let $r\in H^1_{\mathrm{per}}(Y)$ denote the unique solution to \eqref{r problem}, and let $\hat{r}\in H^1_{\mathrm{per},0}(Y)$ be given by \eqref{rhat defn}. Then, there exists a constant $C_0 > 0$ such that for any $\alpha \in (0,C_0)$ it is true that if $R_h$ is a finite-dimensional closed linear subspace of $H^1_{\mathrm{per},0}(Y)$ with the property
\begin{align}\label{ch of psi defn}
\inf_{v_h\in R_h}\frac{\|\psi-v_h\|_{H^1(Y)}}{\|\psi\|_{H^2(Y)}}\leq \alpha\qquad \forall \psi\in H^2_{\mathrm{per},0}(Y)\backslash\{0\},
\end{align}
then there exists a unique $\hat{r}_h\in R_h$ such that 
\begin{align}\label{rhhat prob}
a(\hat{r}_h,v_h) = \int_Y (b-\mathrm{div}(A))\cdot \nabla v_h\qquad \forall v_h\in  R_h,
\end{align}
and setting $r_h :=1+\hat{r}_h$, there holds
\begin{align}\label{err bd rh}
\|r-r_h\|_{L^2(Y)} + \alpha \|r-r_h\|_{H^1(Y)} \leq C\alpha \inf_{v_h\in R_h} \|\hat{r}-v_h\|_{H^1(Y)}
\end{align}
for some constant $C>0$ depending only on $(A,b)$ and $n$. 
\end{theorem}

\begin{proof}
Let $\alpha\in (0,C_0)$, where $C_0>0$ will be chosen later (see \eqref{C0 choice}). Let $R_h$ be a finite-dimensional closed linear subspace of $H^1_{\mathrm{per},0}(Y)$ satisfying \eqref{ch of psi defn}. 

\medskip
\textit{Uniqueness of $\hat{r}_h$}: We show that \eqref{rhhat prob} can have at most one solution $\hat{r}_h\in R_h$. Before we start, note that in view of Proposition \ref{Prop: r}, there exist constants $r_0,r_1>0$ such that $r_0\leq r\leq r_1$ in $\R^n$. Now suppose $\hat{r}_{h}^{(1)}\in R_h$ and $\hat{r}_h^{(2)}\in R_h$ are two solutions to \eqref{rhhat prob}, and set $z_h:=\hat{r}_{h}^{(1)}-\hat{r}_{h}^{(2)}$. Noting that $z_h\in R_h \subset H^1_{\mathrm{per},0}(Y)$ and $a(z_h,v_h) = 0$ for all $v_h\in R_h$, we have by \eqref{a Garding} that
\begin{align}\label{zh H1 bd}
\|z_h\|_{H^1(Y)}^2 \leq \frac{2}{\lambda}\left( a(z_h,z_h) + C_2 \|z_h\|_{L^2(Y)}^2\right) \leq \frac{2 C_2}{\lambda} \|z_h\|_{L^2(Y)}\|z_h\|_{H^1(Y)}.
\end{align}
Since $\frac{z_h}{r}\in L^2_{\mathrm{per}}(Y)$ and $\int_Y z_h = 0$, we have by Proposition \ref{Prop: r} that the problem
\begin{align}\label{psiprob}
-A:D^2 \psi - b\cdot \nabla \psi = \frac{z_h}{r}\quad\text{in }Y,\qquad \psi\text{ is $Y$-periodic},\qquad \int_{Y} \psi = 0
\end{align} 
has a unique solution $\psi\in H^1_{\mathrm{per},0}(Y)$, i.e., $a(v,\psi) = \int_Y \frac{z_h}{r}v$ for any $v\in H^1_{\mathrm{per},0}(Y)$. Note that $\psi\in H^2_{\mathrm{per},0}(Y)$ and $\|\psi\|_{H^2(Y)}\leq C_3 \|\frac{z_h}{r}\|_{L^2(Y)}$ for some constant $C_3 > 0$; see \cite{GT01}. Then, we find that
\begin{align}\label{zhbd}
\begin{split}
\|z_h\|_{L^2(Y)}^2\leq r_1 \int_Y \frac{|z_h|^2}{r} = r_1 a(z_h,\psi) &= r_1 \inf_{v_h\in R_h} a(z_h,\psi-v_h) \\ &\leq r_1 C_1 \alpha \|z_h\|_{H^1(Y)}\|\psi\|_{H^2(Y)} \\ 
&\leq  \frac{r_1 C_1 C_3}{r_0} \alpha \|z_h\|_{H^1(Y)}\|z_h\|_{L^2(Y)},
\end{split}
\end{align} 
where we used the bounds on $r$, the properties of $\psi$ and $z_h$, and \eqref{a bddness}. Combining this estimate with \eqref{zh H1 bd}, we obtain
\begin{align*}
\|z_h\|_{H^1(Y)} \leq \frac{2C_2}{\lambda}\|z_h\|_{L^2(Y)}\leq \frac{2r_1 C_1 C_2 C_3}{\lambda r_0} \alpha \|z_h\|_{H^1(Y)}.
\end{align*}
Let $C_0>0$ be chosen as 
\begin{align}\label{C0 choice}
C_0 := \frac{\lambda r_0}{2r_1 C_1 C_2 C_3}.
\end{align}
Then, using that $\alpha < C_0$, there holds $\frac{2r_1 C_1 C_2 C_3}{\lambda r_0} \alpha < 1$ and hence, $z_h = 0$, i.e., there is at most one solution to \eqref{rhhat prob}. 

\medskip
\textit{Existence of $\hat{r}_h$}: As $R_h$ is finite-dimensional, uniqueness implies existence of a solution $\hat{r}_h\in R_h$ to \eqref{rhhat prob}. 

\medskip
\textit{Error bound}: It remains to show the error bound \eqref{err bd rh}. First, by \eqref{a Garding}, Galerkin orthogonality, and \eqref{a bddness}, we have that
\begin{align}\label{rh-rhh bd}
\begin{split}
\frac{\lambda}{2} \|\hat{r}-\hat{r}_h\|_{H^1(Y)}^2 
&\leq   \inf_{v_h\in R_h} a(\hat{r}-\hat{r}_h,\hat{r}-v_h) + C_2 \|\hat{r}-\hat{r}_h\|_{L^2(Y)}^2 \\
&\leq \left(C_1 \inf_{v_h\in R_h}\|\hat{r}-v_h\|_{H^1(Y)} +C_2\|\hat{r}-\hat{r}_h\|_{L^2(Y)}\right)\|\hat{r}-\hat{r}_h\|_{H^1(Y)}.
\end{split}
\end{align}
Next, by considering \eqref{psiprob} with $z_h$ replaced by $\hat{r}-\hat{r}_h$ and arguing as in \eqref{zhbd} with $z_h$ replaced by $\hat{r}-\hat{r}_h$, we find that
\begin{align}\label{L2 by H1}
\|\hat{r}-\hat{r}_h\|_{L^2(Y)}\leq \frac{r_1 C_1 C_3}{r_0}\alpha \|\hat{r}-\hat{r}_h\|_{H^1(Y)}.
\end{align}
Combining \eqref{rh-rhh bd} and \eqref{L2 by H1}, we obtain
\begin{align*}
\left(\frac{\lambda}{2}-\frac{r_1 C_1 C_2 C_3}{r_0}\alpha\right)\|\hat{r}-\hat{r}_h\|_{H^1(Y)} \leq C_1 \inf_{v_h\in R_h}\|\hat{r}-v_h\|_{H^1(Y)}.
\end{align*}
Finally, noting that $\frac{\lambda}{2}-\frac{r_1 C_1 C_2 C_3}{r_0}\alpha > 0$ by $\alpha < C_0$ and \eqref{C0 choice}, observing that $r-r_h = \hat{r}-\hat{r}_h$, and in view of \eqref{L2 by H1}, we conclude that \eqref{err bd rh} holds.
\end{proof}

As an example, if $R_h$ in Theorem \ref{Thm: Approx r in set A} is chosen to be the space of continuous piecewise affine zero-mean functions on a shape-regular triangulation $\calT_h$ of $\overline{Y}$ into triangles with longest edge $h>0$ conforming with the requirement of periodicity, then \eqref{ch of psi defn} holds with $\alpha = \calO(h)$.

A natural question to ask is if we can obtain a near-best approximation result in the $W^{1,p}(Y)$-norm. The following theorem answers this question positively under the additional assumption that $\mathrm{div}(A)\in L^\infty_{\mathrm{per}}(Y;\R^n)$, following arguments similar to \cite{BS08,RS82}. Note that if $(A,b)\in \calA$ and $\mathrm{div}(A)\in L^\infty_{\mathrm{per}}(Y;\R^n)$, then $r\in W^{1,p}_{\mathrm{per}}(Y)$ for all $p<\infty$ by elliptic regularity theory.

\begin{theorem}[$L^p$-estimates for the approximation of \eqref{r problem}]\label{Thm: Lp est r}

Let $p\in (1,\infty)$ and set $t:=\frac{p}{p-1}$. In the situation of Theorem \ref{Thm: Approx r in set A}, if additionally $\mathrm{div}(A)\in L^\infty_{\mathrm{per}}(Y;\R^n)$, $R_h \subset W^{1,\infty}_{\mathrm{per},0}(Y)$, and
\begin{align*}
\inf_{v_h\in R_h}\frac{\|\psi-v_h\|_{W^{1,t}(Y)}}{\|\psi\|_{W^{2,t}(Y)}}\leq \alpha\qquad \forall \psi\in W^{2,t}_{\mathrm{per},0}(Y)\backslash\{0\},
\end{align*}
then, for $\alpha > 0$ sufficiently small, we have the following bound
\begin{align}\label{Lp err bd rh}
\|r-r_h\|_{L^p(Y)} + \alpha \|r-r_h\|_{W^{1,p}(Y)} \leq C\alpha \inf_{v_h\in R_h} \|\hat{r}-v_h\|_{W^{1,p}(Y)}
\end{align}
for some constant $C>0$ depending only on $(A,b)$ and $n$. 
\end{theorem}

\begin{proof}
First, observe that $\xi:=\left\lvert \hat{r}-\hat{r}_h\right\rvert^{p-1}\mathrm{sign}(\hat{r}-\hat{r}_h)\in L^\infty_{\mathrm{per}}(Y)$ as $\hat{r},\hat{r}_h\in L^\infty_{\mathrm{per}}(Y)$.
By Proposition \ref{Prop: r}, there exists a unique solution $\psi\in H^1_{\mathrm{per},0}(Y)$ to the problem
\begin{align*}
-A:D^2 \psi - b\cdot \nabla \psi = \xi - \int_Y \xi r\quad\text{in }Y,\qquad \psi\text{ is $Y$-periodic},\qquad \int_{Y} \psi = 0.
\end{align*}
Note that $\psi\in W^{2,t}_{\mathrm{per}}(Y)$ and $\|\psi\|_{W^{2,t}(Y)}\leq C \|\xi - \int_Y \xi r\|_{L^t(Y)}$ for some constant $C>0$; see \cite{GT01}. Using that $\int_Y(\hat{r}-\hat{r}_h) = 0$, we find that
\begin{align*}
\|\hat{r}-\hat{r}_h\|_{L^p(Y)}^p =  a(\hat{r}-\hat{r}_h,\psi) &= \inf_{v_h\in R_h} a(\hat{r}-\hat{r}_h,\psi-v_h)
\\ &\lesssim  \|\hat{r}-\hat{r}_h\|_{W^{1,p}(Y)}\inf_{v_h\in R_h} \|\psi-v_h\|_{W^{1,t}(Y)} 
\\ &\lesssim  \alpha \|\hat{r}-\hat{r}_h\|_{W^{1,p}(Y)} \left\|\xi - \int_Y \xi r\right\|_{L^t(Y)}
\\ &\lesssim  \alpha\|\hat{r}-\hat{r}_h\|_{W^{1,p}(Y)} \|\xi\|_{L^t(Y)},
\end{align*}
where the constants absorbed in ``$\lesssim$" only depend on $(A,b)$ and $n$. Since $\|\xi\|_{L^t(Y)} = \|\hat{r}-\hat{r}_h\|_{L^p(Y)}^{p-1}$, we deduce that
\begin{align}\label{Lp errbd}
\|r-r_h\|_{L^p(Y)} = \|\hat{r}-\hat{r}_h\|_{L^p(Y)} \leq C \alpha\|\hat{r}-\hat{r}_h\|_{W^{1,p}(Y)}
\end{align}
for some constant $C>0$ only depending on $(A,b)$ and $n$. 

Next, let $z\in H^1_{\mathrm{per},0}(Y)$ denote the unique solution to
\begin{align*}
-\nabla\cdot(A\nabla z) = \nabla\cdot ((\mathrm{div}(A)-b) r_h)\quad\text{in }Y,\qquad z\text{ is $Y$-periodic},\qquad \int_{Y} z = 0,
\end{align*} 
and note that $z\in W^{1,p}_{\mathrm{per}}(Y)$. Observe that $\hat{r}$ solves a similar problem with $r_h$ replaced by $r$. In particular, 
\begin{align*}
\|\hat{r}-z\|_{W^{1,p}(Y)}\lesssim \|(\mathrm{div}(A)-b) (r-r_h)\|_{L^p(Y)} \lesssim \|r-r_h\|_{L^p(Y)} \lesssim \alpha \|\hat{r}-\hat{r}_h\|_{W^{1,p}(Y)}.
\end{align*}
Note that $\hat{r}_h\in R_h$ is the $\tilde{a}$-orthogonal projection of $z$ onto $R_h$, where $\tilde{a}(v_1,v_2):=\int_Y A\nabla v_1\cdot \nabla v_2$. Hence, by standard finite element theory,
\begin{align*}
\|z-\hat{r}_h\|_{W^{1,p}(Y)}\lesssim \inf_{v_h\in R_h} \|z-v_h\|_{W^{1,p}(Y)} \lesssim  \|\hat{r}-z\|_{W^{1,p}(Y)} +\inf_{v_h\in R_h} \|\hat{r}-v_h\|_{W^{1,p}(Y)}.
\end{align*}
By combining with the previous estimate, we obtain
\begin{align*}
\|\hat{r}-\hat{r}_h\|_{W^{1,p}(Y)} \lesssim \alpha \|\hat{r}-\hat{r}_h\|_{W^{1,p}(Y)} + \inf_{v_h\in R_h} \|\hat{r}-v_h\|_{W^{1,p}(Y)}.
\end{align*}
For $\alpha > 0$ sufficiently small, we can absorb the first term on the right-hand side into the left-hand side and obtain
\begin{align*}
\|r-r_h\|_{W^{1,p}(Y)} = \|\hat{r}-\hat{r}_h\|_{W^{1,p}(Y)} \leq C  \inf_{v_h\in R_h} \|\hat{r}-v_h\|_{W^{1,p}(Y)}
\end{align*}
for some constant $C>0$. Because of \eqref{Lp errbd}, we conclude that \eqref{Lp err bd rh} holds.
\end{proof}

\subsection{Setting $\calB$}\label{Sec: setting B}

In this section, we study well-posedness and the finite element approximation of solutions to the FPK problem \eqref{q problem} for the case $(A,b)\in \calB$, i.e., we now merely assume that
\begin{align*}
A\in L^\infty_{\mathrm{per}}(Y;\R^{n\times n}_{\mathrm{sym}}),\qquad b\in L^\infty_{\mathrm{per}}(Y;\R^n),
\end{align*}
together with uniformly ellipticity (see \eqref{uniform ell}), and that the Cordes-type condition 
\begin{align}\label{mod Cordes again}
\exists\, \delta\in \left(\frac{n}{n+\pi^2},1\right]:\qquad\frac{\lvert A\rvert^2 + \lvert b\rvert^2}{(\mathrm{tr}(A))^2} \leq \frac{1}{n-1+\delta}\quad\text{a.e. in }\R^n
\end{align}
is satisfied. In this case, the problem \eqref{q problem} cannot be reformulated in divergence-form and we will seek (so-called \textit{very weak}) solutions to \eqref{q problem} in $L^2$, i.e., $u\in L^2_{\mathrm{per}}(Y)$ such that
\begin{align*}
\int_Y uA:D^2 \fhi + \int_Y u b\cdot \nabla \fhi = 0\qquad \forall \fhi\in H^2_{\mathrm{per}}(Y).
\end{align*}
Let us introduce the renormalization function
\begin{align}\label{gam}
\gamma := \frac{\mathrm{tr}(A)}{\lvert A\rvert^2 + \lvert b\rvert^2},
\end{align}
and the renormalized coefficients
\begin{align}\label{At Bt}
\tilde{A}:=\gamma A,\qquad \tilde{b}:=\gamma b.
\end{align}
Noting that $\mathrm{tr}(A)$ is the sum of eigenvalues of $A$, and $\lvert A\rvert^2$ is the sum of squared eigenvalues of $A$, we make the following observation.
\begin{remark}[Properties of $\gamma$]\label{Rk: prop gam}
If $(A,b)\in \calB$, then $\gamma\in L^\infty_{\mathrm{per}}(Y)$ and 
\begin{align*}
\gamma_0 := \frac{n\lambda}{n\Lambda^2 + \|b\|_{L^\infty(Y;\R^n)}^2} \leq \gamma \leq \frac{\Lambda}{\lambda^2}\quad\text{a.e. in }\R^n.
\end{align*}
In particular, $\gamma$ is positive almost everywhere.
\end{remark}
We then consider the renormalized FPK problem
\begin{align}\label{q problem renorm}
-D^2:(\tilde{A}\tilde{u}) + \nabla\cdot (\tilde{b}\tilde{u}) = 0\quad\text{in } Y,\qquad \tilde{u}\text{ is $Y$-periodic},
\end{align}
and make the following observation.

\begin{remark}[Relationship between the original and renormalized FPK problem]\label{Rk: rel ren orig}
Let $\mathbb{L}$ denote the set of solutions in $L^2_{\mathrm{per}}(Y)$ to the FPK problem \eqref{q problem}. Then,
\begin{align*}
\mathbb{L} = \{\left.\gamma \tilde{u}\, \right\rvert \tilde{u}\in \tilde{\mathbb{L}}\},
\end{align*}
where $\tilde{\mathbb{L}}$ denotes the set of solutions in $L^2_{\mathrm{per}}(Y)$ to the renormalized FPK problem \eqref{q problem renorm}.
\end{remark}

Because of Remark \ref{Rk: rel ren orig}, we will focus our attention on the renormalized FPK problem \eqref{q problem renorm}. We will begin by analyzing the well-posedness of \eqref{q problem renorm}. 

\subsubsection{Well-posedness}

The key consequence of the Cordes-type condition \eqref{mod Cordes again} is captured in the following lemma.
\begin{lemma}[Consequences of condition \eqref{mod Cordes again}]\label{Lmm: cons of Cor}
Let $(A,b)\in \calB$. Let $(\tilde{A},\tilde{b})$ denote the pair of renormalized coefficients given by \eqref{gam} and \eqref{At Bt}. Then, the following assertions hold:
\begin{itemize}
\item[(i)] We have the bound
\begin{align*}
\left\lvert \tilde{A}-I_n\right\rvert^2 + \left\lvert \tilde{b}\right\rvert^2 \leq 1 - \delta\quad\text{a.e. in }\R^n.
\end{align*}
\item[(ii)] There exists a constant $\kappa = \kappa(\delta,n)\in (0,1]$ such that
\begin{align*}
\left\|\tilde{A}:Dw + \tilde{b}\cdot w - \nabla\cdot w\right\|_{L^2(Y)}\leq \sqrt{1-\kappa}\,\|D w\|_{L^2(Y;\R^{n\times n})}
\end{align*}
for any $w\in H^1_{\mathrm{per},0}(Y;\R^n)$.
\end{itemize}
\end{lemma}

\begin{proof}
Using that $\tilde{A}=\gamma A$ and $\tilde{b} = \gamma b$, we compute
\begin{align*}
\left\lvert \tilde{A}-I_n\right\rvert^2 + \left\lvert \tilde{b}\right\rvert^2 &= n + \gamma^2 \left( \lvert A\rvert^2 + \lvert b\rvert^2 \right) - 2 \gamma\, \mathrm{tr}(A) \\ &= n - \gamma\, \mathrm{tr}(A) \\ &= n - \frac{(\mathrm{tr}(A))^2}{\lvert A\rvert^2 + \lvert b\rvert^2} \leq 1-\delta,
\end{align*}
where we have used \eqref{mod Cordes again} in the final step. This completes the proof of (i) and it remains to show (ii). To this end, let us first note that by Poincar\'{e}'s inequality (see Theorem 3.2 in \cite{Beb03}) there holds
\begin{align}\label{Poin}
\|w\|_{L^2(Y;\R^n)}\leq \frac{\sqrt{n}}{\pi} \|Dw\|_{L^2(Y;\R^{n\times n})}\qquad \forall w\in H^1_{\mathrm{per},0}(Y;\R^n).
\end{align}
Using \eqref{Poin} and (i), we find that
\begin{align*}
\left\|\tilde{A}:Dw + \tilde{b}\cdot w - \nabla\cdot w\right\|_{L^2(Y)}^2 &= \left\|(\tilde{A}-I_n):Dw + \tilde{b}\cdot w\right\|_{L^2(Y)}^2 \\
&\leq \int_Y \left(\left\lvert \tilde{A}-I_n\right\rvert^2 + \left\lvert \tilde{b}\right\rvert^2\right)\left(\lvert Dw\rvert^2 +\lvert w\rvert^2\right) \\
&\leq (1-\delta) \left(\|Dw\|_{L^2(Y;\R^{n\times n})}^2 + \|w\|_{L^2(Y;\R^n)}^2\right) \\
&\leq (1-\delta)\left(1+\frac{n}{\pi^2}\right)\|Dw\|_{L^2(Y;\R^{n\times n})}^2 \\ &= (1-\kappa)\|Dw\|_{L^2(Y;\R^{n\times n})}^2
\end{align*}
for any $w\in H^1_{\mathrm{per},0}(Y;\R^n)$, where $\kappa := (\delta - \frac{n}{n+\pi^2})\frac{n+\pi^2}{\pi^2}$. Note that $\kappa \in (0,1]$ as $\frac{n}{n+\pi^2} < \delta \leq 1$.
\end{proof}

\begin{remark}
If $\lvert b \rvert = 0$ a.e., then Lemma \ref{Lmm: cons of Cor} holds with $\kappa = \delta$, and we can relax the range of $\delta$ in \eqref{mod Cordes again} to $\delta \in (0,1]$.
\end{remark}

\begin{remark}[Another bound]\label{Rk: cons of LmmCor}
In the situation of Lemma \ref{Lmm: cons of Cor} there holds 
\begin{align*}
\left\|\tilde{A}:D^2 \fhi + \tilde{b}\cdot \nabla \fhi - \Delta \fhi\right\|_{L^2(Y)}\leq \sqrt{1-\kappa}\,\|\Delta\fhi\|_{L^2(Y)}\qquad \forall \fhi \in H^2_{\mathrm{per}}(Y),
\end{align*}
where $\kappa\in (0,1]$ is the constant from Lemma \ref{Lmm: cons of Cor}(ii). Indeed, this inequality follows from choosing $w = \nabla \fhi\in H^1_{\mathrm{per},0}(Y;\R^n)$ in Lemma \ref{Lmm: cons of Cor}(ii) and using the fact that that $\|D^2 \fhi\|_{L^2(Y)} = \|\Delta \fhi\|_{L^2(Y)}$ for any $\fhi \in H^2_{\mathrm{per}}(Y)$.
\end{remark}
Next, let us observe that any $\tilde{u}\in L^2_{\mathrm{per}}(Y)$ can be written as $\tilde{u} = c - \Delta \psi_c$ for some unique $c\in \R$ (namely $c = \int_Y \tilde{u}$) and unique $\psi_c \in H^2_{\mathrm{per},0}(Y)$. Inserting this ansatz into \eqref{q problem renorm} leads to the problem of finding $c\in \R$ and $\psi_c \in H^2_{\mathrm{per},0}(Y)$ such that 
\begin{align}\label{B1 prob}
B_1(\psi_c,\fhi) = c\int_Y \left(\tilde{A}:D^2 \fhi +   \tilde{b}\cdot \nabla \fhi\right)\qquad \forall \fhi\in H^2_{\mathrm{per},0}(Y),
\end{align}
where $B_1(\cdot,\cdot):H^2_{\mathrm{per},0}(Y)\times H^2_{\mathrm{per},0}(Y)\rightarrow \R$ is the bilinear form defined by
\begin{align}\label{B1 defn}
B_1(\fhi_1,\fhi_2) := \int_Y \Delta \fhi_1\left(\tilde{A}:D^2 \fhi_2 + \tilde{b}\cdot \nabla \fhi_2\right).
\end{align}
We can show the following result:
\begin{theorem}[Analysis of \eqref{q problem} in setting $\calB$]\label{Thm: ana in set B}
Let $(A,b)\in \calB$. Let $\gamma$ denote the renormalization function and $(\tilde{A},\tilde{b})$ be the pair of renormalized coefficients given by \eqref{gam}, \eqref{At Bt}. Further, let $B_1:H^2_{\mathrm{per},0}(Y)\times H^2_{\mathrm{per},0}(Y)\rightarrow \R$ denote the bilinear form defined in \eqref{B1 defn}.
\begin{itemize}
\item[(i)] For any $c\in \R$ there exists a unique solution $\psi_c\in H^2_{\mathrm{per},0}(Y)$ to \eqref{B1 prob}, and we have that $\psi_c = c \psi_1$.
\item[(ii)] The function $\tilde{r}:=1-\Delta \psi_1$ is the unique solution in $L^2_{\mathrm{per}}(Y)$ to the problem 
\begin{align}\label{rt prob}
-D^2:\left(\tilde{A}\tilde{r}\right) + \nabla\cdot \left(\tilde{b}\tilde{r}\right) = 0\quad\text{in } Y,\qquad \tilde{r}\text{ is $Y$-periodic},\qquad \int_Y \tilde{r} = 1.
\end{align}
Further, we have that $\tilde{r}\geq 0$ a.e. in $\R^n$.
\item[(iii)] The set $\mathbb{L}$ of all solutions in $L^2_{\mathrm{per}}(Y)$ to the FPK problem \eqref{q problem} is given by $\mathbb{L} = \{\left.c\gamma\tilde{r}\,\right\rvert c\in \R\}$.
\end{itemize}
\end{theorem}

\begin{proof}
(i) We show that the Lax--Milgram theorem applies. Clearly, $B_1$ is a bounded bilinear form on $H^2_{\mathrm{per},0}(Y)$ and the right-hand side in \eqref{B1 prob} defines a bounded linear functional on $H^2_{\mathrm{per},0}(Y)$. It remains to show that $B_1$ is coercive on $H^2_{\mathrm{per},0}(Y)$. By Remark \ref{Rk: cons of LmmCor}, we have for any $\fhi \in H^2_{\mathrm{per},0}(Y)$ that
\begin{align*}
B_1(\fhi,\fhi) &= \|\Delta \fhi\|_{L^2(Y)}^2 + \int_Y \Delta \fhi \left(\tilde{A}:D^2 \fhi + \tilde{b}\cdot \nabla \fhi - \Delta \fhi \right) 
\\ &\geq (1-\sqrt{1-\kappa})\|\Delta \fhi\|_{L^2(Y)}^2,
\end{align*}
where $\kappa\in (0,1]$ is the constant from Lemma \ref{Lmm: cons of Cor}. The proof is concluded by observing that $1-\sqrt{1-\kappa}>0$ and that $\|\fhi\|_{H^2(Y)}\leq C \|\Delta \fhi\|_{L^2(Y)}$ for any $\fhi\in H^2_{\mathrm{per},0}(Y)$, where $C > 0$ is a constant. The fact that $\psi_c = c \psi_1$ follows from the linearity of $B_1$ in its first argument and the uniqueness of $\psi_c$.

(ii) Clearly, $\tilde{r}:=1-\Delta \psi_1 \in L^2_{\mathrm{per}}(Y)$ is a solution to \eqref{rt prob}. Conversely, any solution $\tilde{r}\in L^2_{\mathrm{per}}(Y)$ to \eqref{rt prob} can be written as $\tilde{r} = 1 - \Delta \psi$ for some unique $\psi\in H^2_{\mathrm{per},0}(Y)$. Clearly, $\psi$ must satisfy \eqref{B1 prob} with $c = 1$ and hence, by (i), $\psi = \psi_1$. We now show that $\tilde{r}\geq 0$ a.e. in $\R^n$, using a mollification argument similar to \cite{BRS12}. For $k\in \N$, we set 
\begin{align*}
\tilde{A}_{(k)}:= (\tilde{a}_{ij}\ast w_k)_{1\leq i,j\leq n}\in C^{\infty}_{\mathrm{per}}(Y;\R^{n\times n}_{\mathrm{sym}}),\qquad \tilde{b}_{(k)}:=(\tilde{b}_i\ast w_k)_{1\leq i\leq n}\in C^{\infty}_{\mathrm{per}}(Y;\R^{n}),
\end{align*}
where $w_k:=k^n w(k\,\cdot)$ for some $w\in C_c^{\infty}(\R^n)$ with $w\geq 0$ in $\R^n$ and $\int_{\R^n} w = 1$. 

Note that $\tilde{A}_{(k)}$ is uniformly elliptic. In particular, there exists a positive solution $\tilde{r}_k\in C^{\infty}_{\mathrm{per}}(Y;(0,\infty))$ to $-D^2:(\tilde{A}_{(k)}\tilde{r}_k) + \nabla\cdot (\tilde{b}_{(k)}\tilde{r}_k) = 0$ in $Y$ with $\int_Y \tilde{r}_k = 1$; see, e.g., \cite{BLP11}. Let $\psi_k\in H^2_{\mathrm{per},0}(Y)$ be such that $\Delta \psi_k = 1-\tilde{r}_k$. We are going to show that $\|\Delta \psi_k\|_{L^2(Y)}$ is uniformly bounded. To this end, first note that
\begin{align*}
\lvert I_n - \tilde{A}_{(k)}(y)\rvert^2 + \lvert \tilde{b}_{(k)}(y)\rvert^2 &= \left\lvert \int_{\R^n} [I_n - \tilde{A}(y-\cdot)]w_k \right\rvert^2 + \left\lvert \int_{\R^n} \tilde{b}(y-\cdot)w_k \right\rvert^2 \\
&\leq \int_{\R^n} \left(\lvert I_n - \tilde{A}(y-\cdot)\rvert^2 + \lvert \tilde{b}(y-\cdot)\rvert^2\right)w_k \leq 1-\delta
\end{align*}
for any $y\in \R^n$ by Lemma \ref{Lmm: cons of Cor}(i), and hence, the bound from Remark \ref{Rk: cons of LmmCor} still holds when $(\tilde{A},\tilde{b})$ is replaced by $(\tilde{A}_{(k)},\tilde{b}_{(k)})$. Then, we find that
\begin{align*}
(1-\sqrt{1-\kappa})\|\Delta \psi_k\|_{L^2(Y)}^2 &\leq \int_Y \Delta \psi_k\left(\tilde{A}_{(k)}:D^2 \psi_k + \tilde{b}_{(k)}\cdot \nabla \psi_k\right) \\ &= \int_Y \left(\tilde{A}_{(k)}:D^2 \psi_k+\tilde{b}_{(k)}\cdot\nabla \psi_k\right)\leq C\|\Delta \psi_k\|_{L^2(Y)}
\end{align*}
for some constant $C = C(n,\lambda,\Lambda,\|b\|_{L^\infty(Y)})>0$ independent of $k$, where we have used that $\lvert \tilde{b}_{(k)} \rvert \leq \|\tilde{b}\|_{L^\infty(Y;\R^n)}$ and $\lvert \tilde{A}_{(k)}\rvert \leq \sqrt{n}\frac{\Lambda}{\lambda}$ in $\R^n$ (as $\lvert \tilde{A}\rvert = \gamma \lvert A\rvert \leq \sqrt{n}\frac{\Lambda}{\lambda}$ a.e. in $\R^n$). In particular, $\|\Delta \psi_k\|_{L^2(Y)}$ is uniformly bounded, and hence, $\|\tilde{r}_k\|_{L^2(Y)} = \|1-\Delta \psi_k\|_{L^2(Y)}$ is uniformly bounded. 

Therefore, there exists an $\tilde{r}_0\in L^2_{\mathrm{per}}(Y)$ such that, upon passing to a subsequence (not indicated), $\tilde{r}_k\weak \tilde{r}_0$ weakly in $L^2(Y)$. There holds $\int_Y \tilde{r}_0 = 1$ and, using that $\tilde{A}_{(k)}\rightarrow \tilde{A}$ in the $L^2(Y;\R^{n\times n})$-norm and $\tilde{b}_{(k)}\rightarrow \tilde{b}$ in the $L^2(Y;\R^{n})$-norm as $k\rightarrow \infty$, we have  for any $\fhi\in C^{\infty}_{\mathrm{per}}(Y)$ that 
\begin{align*}
\int_Y \tilde{r}_0(\tilde{A}:D^2 \fhi + \tilde{b}\cdot \nabla \fhi) = \lim_{k\rightarrow \infty} \int_Y \tilde{r}_k(\tilde{A}_{(k)}:D^2 \fhi + \tilde{b}_{(k)}\cdot\nabla \fhi)  = 0,
\end{align*}
i.e., $\tilde{r}_0\in L^2_{\mathrm{per}}(Y)$ is a solution to \eqref{rt prob}. Since $\tilde{r}$ is the unique solution to \eqref{rt prob} in $L^2_{\mathrm{per}}(Y)$, and since $\tilde{r}_k > 0$ in $\R^n$, we conclude that $\tilde{r}=\tilde{r}_0\geq 0$ a.e. in $\R^n$.

(iii) In view of our discussion above Theorem \ref{Thm: ana in set B}, and the results (i)--(ii), it is easy to confirm that the set $\tilde{\mathbb{L}}$ of all solutions in $L^2_{\mathrm{per}}(Y)$ to the renormalized FPK problem \eqref{q problem renorm} is given by 
\begin{align*}
\tilde{\mathbb{L}} = \{c-\Delta \psi_c \,|\, c\in \R\} = \{c(1-\Delta \psi_1)\,|\, c\in \R\} = \{c\tilde{r}\,|\, c\in \R\}.
\end{align*}
By Remark \ref{Rk: rel ren orig}, it follows that $\mathbb{L} = \{\gamma \tilde{u}\,|\,\tilde{u}\in \tilde{\mathbb{L}}\} = \{\left.c\gamma\tilde{r}\,\right\rvert c\in \R\}$.
\end{proof}

We immediately obtain the following consequences:

\begin{corollary}[Invariant measure in setting $\calB$]\label{Cor: r in setB}

For any $(A,b)\in \calB$, there exists a unique solution $r\in L^2_{\mathrm{per}}(Y)$ to 
\begin{align}\label{r prob setB}
-D^2:\left(Ar\right) + \nabla\cdot \left( br\right) = 0\quad\text{in } Y,\qquad r\text{ is $Y$-periodic},\qquad \int_Y r = 1,
\end{align}
and there holds $r\geq 0$ a.e. in $\R^n$. Further, we can obtain that 
\begin{align*}
r = \frac{1}{\int_Y \gamma\tilde{r}}\gamma \tilde{r}
\end{align*}
with $\gamma$ defined in \eqref{gam} and $\tilde{r}$ defined in Theorem \ref{Thm: ana in set B}(ii). 
\end{corollary}

\subsubsection{Finite element approximation of \eqref{rt prob}}\label{Sec: FEM rt}

Let us recall from Theorem \ref{Thm: ana in set B} that $u\in L^2_{\mathrm{per}}(Y)$ is a solution to \eqref{q problem}, if and only if $u = c\gamma \tilde{r}$ for some $c\in \R$. Hence, once we know the unique solution $\tilde{r}\in L^2_{\mathrm{per}}(Y)$ to \eqref{rt prob}, we know all solutions to \eqref{q problem}. We now discuss the finite element approximation of $\tilde{r}$, inspired by ideas from \cite{Spr24}.

One could approximate $\tilde{r}$ via an $H^2_{\mathrm{per},0}(Y)$-conforming finite element method for the approximation of $\psi_1\in H^2_{\mathrm{per},0}(Y)$ from Theorem \ref{Thm: ana in set B}(i). More conveniently, we can avoid using $H^2$-conforming methods by using the following scheme based on an $H^1_{\mathrm{per},0}(Y;\R^n)$-conforming finite element method to approximate an auxiliary function $\rho\in H^1_{\mathrm{per},0}(Y;\R^n)$ satisfying $\nabla\cdot \rho = \Delta \psi_1$. To this end, we introduce the bilinear form $B_2(\cdot,\cdot):H^1_{\mathrm{per},0}(Y;\R^n)\times H^1_{\mathrm{per},0}(Y;\R^n)\rightarrow \R$ given by
\begin{align*}
B_2(v,w) := \int_Y (\nabla\cdot v) ( \tilde{A}:Dw + \tilde{b}\cdot w) + \frac{1}{2} \int_Y (Dv-(Dv)^{\mathrm{T}}): (Dw-(Dw)^{\mathrm{T}})
\end{align*}
for $v,w\in H^1_{\mathrm{per},0}(Y;\R^n)$. Note that $B_2(\nabla \fhi_1,\nabla \fhi_2) = B_1(\fhi_1,\fhi_2)$ for any $\fhi_1,\fhi_2\in H^2_{\mathrm{per},0}(Y)$. The second integral in the definition of $B_2$ is added to make $B_2$ coercive.

\begin{lemma}[Characterization of $\tilde{r}$]\label{Lmm: char of rt}
Let $(A,b)\in \calB$, and let $(\tilde{A},\tilde{b})$ be the pair of renormalized coefficients given by \eqref{gam}, \eqref{At Bt}. Let $B_2:H^1_{\mathrm{per},0}(Y;\R^n)\times H^1_{\mathrm{per},0}(Y;\R^n)\rightarrow \R$ be defined as above. Then, $B_2$ is a bounded and coercive bilinear form on $H^1_{\mathrm{per},0}(Y;\R^n)$. In particular, there exists a unique $\rho\in H^1_{\mathrm{per},0}(Y;\R^n)$ such that
\begin{align}\label{p prob}
B_2(\rho,w) = \int_Y \left(\tilde{A}:Dw + \tilde{b}\cdot w\right)\qquad \forall w\in H^1_{\mathrm{per},0}(Y;\R^n),
\end{align}
and the unique solution $\tilde{r}\in L^2_{\mathrm{per}}(Y)$ to \eqref{rt prob} is given by $\tilde{r} = 1- \nabla\cdot \rho$.
\end{lemma}

\begin{proof}
First, we show that $B_2$ is coercive.  Using that for any $w\in H^1_{\mathrm{per},0}(Y;\mathbb{R}^n)$, we have that
\begin{align}\label{rel Dw}
\|\nabla\cdot w\|_{L^2(Y)}^2 + \frac{1}{2}\|Dw-(Dw)^{\mathrm{T}}\|_{L^2(Y;\R^{n\times n})}^2 = \|Dw\|_{L^2(Y;\R^{n\times n})}^2,
\end{align}
and using Lemma \ref{Lmm: cons of Cor}(ii), we find that for any $w\in H^1_{\mathrm{per},0}(Y;\mathbb{R}^n)$ there holds
\begin{align*}
B_2(w,w) &= \|D w\|_{L^2(Y;\R^{n\times n})}^2 + \int_Y (\nabla\cdot w)\left(\tilde{A}:Dw +\tilde{b}\cdot w - \nabla\cdot w\right) \\
&\geq (1-\sqrt{1-\kappa})\|Dw\|_{L^2(Y;\R^{n\times n})}^2,
\end{align*} 
where $\kappa = \kappa(\delta,n)\in (0,1]$ is the constant from Lemma \ref{Lmm: cons of Cor}. By noting that $1-\sqrt{1-\kappa}>0$ and that $w\mapsto\|Dw\|_{L^2(Y;\R^{n\times n})}$ defines a norm on $H^1_{\mathrm{per},0}(Y;\R^n)$, this concludes the proof of coercivity of $B_2$. 

Next, we show boundedness of $B_2$. Using that by \eqref{gam}, \eqref{At Bt} and Remark \ref{Rk: prop gam}, we can get that
\begin{align*}
\left\lvert \tilde{A}\right\rvert^2 + \left\lvert \tilde{b}\right\rvert^2 = \gamma^2(\lvert A\rvert^2 + \lvert b\rvert^2) = \gamma\, \mathrm{tr}(A)\leq \frac{\Lambda^2}{\lambda^2}n\quad\text{a.e. in }\R^n,
\end{align*}
and using \eqref{Poin} and \eqref{rel Dw}, we have for any $w_1,w_2\in H^1_{\mathrm{per},0}(Y;\mathbb{R}^n)$ that
\begin{align*}
\lvert B_2(w_1,w_2)\rvert &\leq  \frac{\Lambda}{\lambda}\sqrt{n}\,\|\nabla\cdot w_1\|_{L^2(Y)}\sqrt{\|D w_2\|_{L^2(Y;\R^{n\times n})}^2 + \|w_2\|_{L^2(Y;\R^n)}^2} \\ &\qquad+ \|D w_1\|_{L^2(Y;\R^{n\times n})}\|D w_2\|_{L^2(Y;\R^{n\times n})}\\ &\leq \left(1+\frac{\Lambda}{\lambda}\sqrt{n}\sqrt{1+\frac{n}{\pi^2}}\right)\|D w_1\|_{L^2(Y;\R^{n\times n})}\|D w_2\|_{L^2(Y;\R^{n\times n})},
\end{align*} 
which concludes the proof of the boundedness of $B_2$.

Since the right-hand side in \eqref{p prob} defines a bounded linear functional on $H^1_{\mathrm{per},0}(Y;\R^n)$, we have by the Lax--Milgram theorem that \eqref{p prob} has a unique solution $\rho\in H^1_{\mathrm{per},0}(Y;\R^n)$. Finally, noting that $1-\nabla\cdot \rho\in L^2_{\mathrm{per}}(Y)$, that $\int_Y (1-\nabla\cdot \rho) = 1$, and that for any $\fhi\in H^2_{\mathrm{per},0}(Y)$, there holds 
\begin{align*}
\int_Y (1-\nabla\cdot \rho)\left(\tilde{A}:D^2 \fhi + \tilde{b}\cdot \nabla \fhi\right) = \int_Y \left(\tilde{A}:D^2 \fhi + \tilde{b}\cdot \nabla \fhi\right) - B_2(\rho,\nabla \fhi) = 0
\end{align*}
by \eqref{p prob} with $w = \nabla \fhi$, we conclude that $\tilde{r} = 1-\nabla\cdot \rho$ by uniqueness of $\tilde{r}$.
\end{proof}

In view of Lemma \ref{Lmm: char of rt}, we immediately obtain the following simple method to approximate $\tilde{r}$.

\begin{theorem}[Finite element approximation of \eqref{rt prob}]\label{Thm: FEM rtilde}

Suppose that the assumptions of Lemma \ref{Lmm: char of rt} hold. Further, let $P_h$ be a closed linear subspace of $H^1_{\mathrm{per},0}(Y;\mathbb{R}^n)$. Then, there exists a unique $\rho_h\in P_h$ such that 
\begin{align*}
B_2(\rho_h,w_h) = \int_Y \left(\tilde{A}:Dw_h + \tilde{b}\cdot w_h\right)\qquad \forall w_h\in P_h,
\end{align*}
and setting $\tilde{r}_h:=1-\nabla\cdot \rho_h\in L^2_{\mathrm{per}}(Y)$, there holds
\begin{align}\label{err bd rt approx}
\|\tilde{r}-\tilde{r}_h\|_{L^2(Y)} \leq C \inf_{w_h\in P_h}\|D(\rho-w_h)\|_{L^2(Y;\R^{n\times n})}
\end{align}
for some constant $C = C(\lambda,\Lambda,\delta,n)>0$.
\end{theorem}

\begin{proof}
This follows immediately from the statement and proof of Lemma \ref{Lmm: char of rt} together with a standard Galerkin orthogonality argument. The constant $C>0$ in \eqref{err bd rt approx} can be taken as $C := \frac{1+({\Lambda}/{\lambda})\sqrt{n}\sqrt{1+(n/\pi^2)}}{1-\sqrt{1-\kappa}}$, where $\kappa = \kappa(\delta,n)\in (0,1]$ is the constant from Lemma \ref{Lmm: cons of Cor}. 
\end{proof}

\section{Remarks on Nonhomogeneous Stationary FPK-type Problems}\label{Sec: NonhomFPK}

In this section, we briefly discuss the finite element approximation of nonhomogeneous stationary FPK-type problems subject to periodic boundary conditions, as the ideas from the previous section can be straightforwardly extended to this problem class. For $F\in L^2_{\mathrm{per}}(Y;\R^n)$, let us consider the problem
\begin{align}\label{q problem inhom}
-D^2:(Au) + \nabla\cdot (bu) = \nabla\cdot F \quad\text{in } Y,\qquad u\text{ is $Y$-periodic},
\end{align} 
where $Y:=(0,1)^n$ denotes the unit cell in $\R^n$, and $(A,b)\in \mathcal{A}$ or $(A,b)\in \mathcal{B}$. Recall the settings  $\calA$ (higher regularity) and $\calB$ (Cordes-type) from Section \ref{Sec: Hom FPK}.

\subsection{Setting $\calA$}

First, we note that when $(A,b)\in\calA$, we can rewrite the problem \eqref{q problem inhom} in divergence-form thanks to \eqref{AinW1p}, i.e.,
\begin{align*}
-\nabla\cdot\left(A\nabla u + (\mathrm{div}(A)-b) u\right) = \nabla\cdot F\quad\text{in } Y,\qquad u\text{ is $Y$-periodic}.
\end{align*}
Regarding the uniqueness of solutions, we know from Proposition \ref{Prop: r} that if a solution $u\in H^1_{\mathrm{per}}(Y)$ to \eqref{q problem inhom} exists, then it is unique up to the addition of constant multiples of $r$, where $r$ denotes the unique solution to \eqref{r problem}. Regarding the existence of solutions to \eqref{q problem inhom}, the following result is known in Setting $\calA$; see, e.g., \cite{GT01}.

\begin{proposition}[Well-posedness of \eqref{q problem inhom} in setting $\calA$]\label{Prop: FPK inhom A}
Let $(A,b)\in \mathcal{A}$ and $F\in L^2_{\mathrm{per}}(Y;\R^n)$. Then, there exists a solution $u\in H^1_{\mathrm{per}}(Y)$ to \eqref{q problem inhom}, and $u$ is unique up to the addition of a constant multiple of the unique solution $r$ to \eqref{r problem}.
\end{proposition}
To have a unique solution, let us now restrict our attention to the unique solution $u_{0}\in H^1_{\mathrm{per}}(Y)$ to the following problem 
\begin{align}\label{q0 problem}
-\nabla\cdot\left(A\nabla u_{0} + (\mathrm{div}(A)-b) u_{0}\right) = \nabla\cdot F\;\text{ in } Y,\;\; u_{0}\text{ is $Y$-periodic},\;\; \int_Y u_{0} = 0,
\end{align}
whose existence and uniqueness follows from Proposition \ref{Prop: FPK inhom A}. Then, arguing as in Section \ref{Sec: FEM r Set A}, we obtain the following approximation result, the proof of which is omitted.

\begin{theorem}[Finite element approximation of \eqref{q0 problem}]
Let $(A,b)\in \calA$ and $F\in L^2_{\mathrm{per}}(Y;\R^n)$. Let $u_0\in H^1_{\mathrm{per},0}(Y)$ denote the unique solution to \eqref{q0 problem}, and let $a:H^1_{\mathrm{per},0}(Y)\times H^1_{\mathrm{per},0}(Y) \rightarrow \R$ denote the bilinear form from \eqref{a bil form}. Then, there exists a constant $C_0 > 0$ such that for any $\alpha \in (0,C_0)$ it is true that if $R_h$ is a finite-dimensional closed linear subspace of $H^1_{\mathrm{per},0}(Y)$ with the property
\begin{align*}
\inf_{v_h\in R_h}\frac{\|\psi-v_h\|_{H^1(Y)}}{\|\psi\|_{H^2(Y)}}\leq \alpha\qquad \forall \psi\in H^2_{\mathrm{per},0}(Y)\backslash\{0\},
\end{align*}
then there exists a unique $u_h\in R_h$ such that $a(u_h,v_h) = -\int_Y F\cdot \nabla v_h$ for all $v_h\in  R_h$, and there holds
\begin{align*}
\|u_0-u_h\|_{L^2(Y)} + \alpha \|u_0-u_h\|_{H^1(Y)} \leq C\alpha \inf_{v_h\in R_h} \|u_0-v_h\|_{H^1(Y)}
\end{align*}
for some constant $C>0$ depending only on $(A,b)$ and $n$. 
\end{theorem}
 
If desired, $L^p$ approximation estimates can be derived similarly to Theorem \ref{Thm: Lp est r} under suitable regularity assumptions on $F$.

\subsection{Setting $\calB$}

When $(A,b)\in \calB$, we consider the renormalized problem
\begin{align}\label{q problem inhom ren}
-D^2:(\tilde{A}\tilde{u}) + \nabla\cdot (\tilde{b}\tilde{u}) = \nabla\cdot F \quad\text{in } Y,\qquad \tilde{u}\text{ is $Y$-periodic},
\end{align} 
where the pair of renormalized coefficients $(\tilde{A},\tilde{b})$ is given by \eqref{gam}, \eqref{At Bt}. Similarly to Remark \ref{Rk: rel ren orig}, it is clear that $u\in L^2_{\mathrm{per}}(Y)$ is a solution to the original problem \eqref{q problem inhom} if and only if there is a solution $\tilde{u}\in L^2_{\mathrm{per}}(Y)$ to \eqref{q problem inhom ren} such that $u = \gamma\tilde{u}$, where $\gamma$ is defined in \eqref{gam}. Hence, from now on we focus on the renormalized problem \eqref{q problem inhom ren}.

Regarding uniqueness of solutions, we know from Theorem \ref{Thm: ana in set B} that if a solution $\tilde{u}\in L^2_{\mathrm{per}}(Y)$ to \eqref{q problem inhom ren} exists, then it is unique up to the addition of a constant multiple of $\tilde{r}$, where $\tilde{r}\in L^2_{\mathrm{per}}(Y)$ denotes the unique solution to \eqref{rt prob}. In order to have at most one solution, we therefore focus our attention on the more restricted problem 
\begin{align}\label{qt0 problem}
-D^2:(\tilde{A}\tilde{u}_0) + \nabla\cdot (\tilde{b}\tilde{u}_0) = \nabla\cdot F \quad\text{in } Y,\qquad \tilde{u}_0\text{ is $Y$-periodic},\qquad \int_Y \tilde{u}_0 = 0,
\end{align} 
which indeed has a unique solution, as summarized in the following theorem. 

\begin{theorem}[Well-posedness of \eqref{q problem inhom ren}]\label{Thm: wellp of qinhom ren}
Let $(A,b)\in \calB$ and $F\in L^2_{\mathrm{per}}(Y;\R^n)$. Let $(\tilde{A},\tilde{b})$ be the pair of renormalized coefficients given by \eqref{gam}, \eqref{At Bt}, and let $\tilde{r}\in L^2_{\mathrm{per}}(Y)$ denote the unique solution to \eqref{rt prob}. Further, let $B_2:H^1_{\mathrm{per},0}(Y;\R^n)\times H^1_{\mathrm{per},0}(Y;\R^n)\rightarrow \R$ denote the bilinear form from Lemma \ref{Lmm: char of rt}. Then, the following assertions hold.
\begin{itemize}
\item[(i)] There exists a unique $\tilde{\rho}_0\in H^1_{\mathrm{per},0}(Y;\R^n)$ such that $B_2(\tilde{\rho}_0,w) = -\int_Y F\cdot w$ for all $w\in H^1_{\mathrm{per},0}(Y;\R^n)$.
\item[(ii)] The function $\tilde{u}_0:=-\nabla\cdot \tilde{\rho}_0$ is the unique solution in $L^2_{\mathrm{per}}(Y)$ to \eqref{qt0 problem}.
\item[(iii)] A function $\tilde{u}\in L^2_{\mathrm{per}}(Y)$ is a solution to \eqref{q problem inhom ren} if and only if $\tilde{u} = \tilde{u}_0 + c\tilde{r}$ for some constant $c\in \R$.
\end{itemize}
\end{theorem}

We omit the proof as the results follow immediately from the previous discussion and the coercivity of $B_2$ on $H^1_{\mathrm{per},0}(Y;\R^n)$. Then, arguing as in Section \ref{Sec: FEM rt}, we obtain the following approximation result, the proof of which is omitted.

\begin{theorem}[Finite element approximation of \eqref{qt0 problem}]
Suppose that the assumptions of Theorem \ref{Thm: wellp of qinhom ren} hold. Further, let $P_h$ be a closed linear subspace of $H^1_{\mathrm{per},0}(Y;\mathbb{R}^n)$. Then, there exists a unique $\tilde{\rho}_h\in P_h$ such that $B_2(\tilde{\rho}_h,w_h) = -\int_Y F\cdot w_h$ for all $w_h\in P_h$, and setting $\tilde{u}_h:=-\nabla\cdot \tilde{\rho}_h\in L^2_{\mathrm{per}}(Y)$, there holds
\begin{align*}
\|\tilde{u}_0-\tilde{u}_h\|_{L^2(Y)} \leq C \inf_{w_h\in P_h}\|D(\tilde{\rho}_0-w_h)\|_{L^2(Y;\R^{n\times n})}
\end{align*}
for some constant $C = C(\lambda,\Lambda,\delta,n)>0$.
\end{theorem}

\section{Numerical Approximation of Effective Diffusion Matrices}\label{Sec: SecAbar}

\subsection{Setting $\calA$}\label{Sec: SecAbar A}

First, let us discuss the case $(A,b)\in \calA$. We denote the invariant measure by $r$, i.e., the unique solution to the periodic FPK problem \eqref{r problem} given by Proposition \ref{Prop: r}. As usual, we write $Y:=(0,1)^n$, and we introduce the notations $\langle \fhi \rangle := \int_Y \fhi r$ and $\fhi^{\eps}:=\fhi(\frac{\cdot}{\eps})$ for any $Y$-periodic function $\fhi$ and $\eps>0$.

\subsubsection{The homogenization result}

When $(A,b)\in \calA$, it is known (see, e.g., \cite{JZ23}) that if the drift $b$ satisfies the centering condition
\begin{align}\label{cent cond}
\langle b \rangle = 0,
\end{align}
then for any bounded domain $\Omega\subset \R^n$ with $\partial\Omega \in C^{1,1}$, $f\in L^2(\Omega)$, and $g\in H^2(\Omega)$, the sequence of solutions $(u_{\eps})_{\eps>0}\in H^2(\Omega)$ to 
\begin{align}\label{ueps prob}
\begin{split}
-A^{\eps}:D^2 u_\eps  - \eps^{-1}\,b^\eps\cdot \nabla u_\eps &= f\quad\text{in }\Omega,\\  u_{\eps} &= g\quad\text{on }\partial\Omega,
\end{split}
\end{align}
converges weakly in $H^1(\Omega)$ as $\eps \rightarrow 0$ to the solution $\overline{u}$ of the homogenized problem
\begin{align}\label{u pro}
\begin{split}
-\overline{A}:D^2 \overline{u} &= f\quad\text{in }\Omega,\\  \overline{u} &= g\quad\text{on }\partial\Omega,
\end{split}
\end{align}
where the effective diffusion matrix $\overline{A}\in \R^{n\times n}$ is the symmetric positive definite matrix given by
\begin{align}\label{def Abar}
\overline{A} := \left\langle[I_n+D\chi]A[I_n + (D\chi)^{\mathrm{T}}]\right\rangle,\qquad \chi = (\chi_1,\dots,\chi_n),
\end{align}
with $\chi_j\in H^1_{\mathrm{per}}(Y)$, $1\leq j\leq n$, denoting the solution to 
\begin{align}\label{chij pro}
-A:D^2 \chi_j - b\cdot\nabla \chi_j = b_j\quad\text{in } Y,\qquad \chi_j\text{ is $Y$-periodic},\qquad \int_{Y} \chi_j = 0,
\end{align}
whose existence and uniqueness are guaranteed by Proposition \ref{Prop: r}.

We make it our goal to approximate the effective diffusion matrix $\overline{A}$, a quantity that only depends on $(A,b)$ and is, in particular,  independent of $\Omega,f,g$.

\subsubsection{Approximation of the effective diffusion matrix}

First, we recall that for $(A,b)\in \calA$, we have that $r\in W^{1,p}_{\mathrm{per}}(Y)$ with $p>n$ as in \eqref{AinW1p}. In particular, $r$ is H\"{o}lder continuous. We also recall from Section \ref{Sec: Hom FPK} that we have constructed a finite element method for the approximation of $r$ in the $H^1(Y)$-norm (see Theorem \ref{Thm: Approx r in set A}) and in the $W^{1,p}(Y)$-norm (assuming additionally that $\mathrm{div}(A)\in L^{\infty}$; see Theorem \ref{Thm: Lp est r}). In particular, by Sobolev embedding, we can produce a sequence of approximations $(r_h)_{h>0}$ with $\|r-r_h\|_{L^\infty(Y)}\rightarrow 0$.

In view of the definition \eqref{def Abar} of the effective diffusion matrix, it is thus sufficient for us to approximate the solution $\chi_j$ to \eqref{chij pro} in the $H^1(Y)$-norm. The key difficulty is that for the dual problem
\begin{align*}
-D^2:(A\phi ) + \nabla\cdot (b\phi ) = z\quad\text{in } Y,\qquad \phi\text{ is $Y$-periodic},\qquad \int_{Y} \phi = 0
\end{align*}
with $z\in L^2_{\mathrm{per},0}(Y)$, we generally cannot expect that $\phi\in H^2_{\mathrm{per}}(Y)$ and hence, a duality argument similar to the one in the proof of Theorem \ref{Thm: Approx r in set A} does not work. E.g., when $n=1$, $b = 0$, and $z(y) = \sin(2\pi y)$, then $\phi(y) = \frac{\sin(2\pi y)+c}{4\pi^2 A(y)}$ for some $c\in \R$. This simple example shows that lack of regularity of $A$ limits the regularity of $\phi$.

To overcome this difficulty, we note that the solution $\chi_j$ to \eqref{chij pro} belongs to $H^2_{\mathrm{per}}(Y)$, and thus we can multiply the equation \eqref{chij pro} by the positive H\"{o}lder continuous function $r$ from Proposition \ref{Prop: r} without changing the set of solutions:
\begin{align*}
-rA:D^2 \chi_j - rb\cdot\nabla \chi_j = rb_j\quad\text{in } Y,\qquad \chi_j\text{ is $Y$-periodic},\qquad \int_{Y} \chi_j = 0.
\end{align*}
Since $(A,b)\in \calA$, introducing $\beta:=rb-\mathrm{div}(rA) \in L^p_{\mathrm{per}}(Y;\R^n)$, $p>n$, we can rewrite the equation in divergence-form as
\begin{align}\label{chij with r}
-\nabla\cdot (rA\nabla \chi_j) - \beta\cdot \nabla \chi_j = rb_j\quad\text{in } Y,\qquad \chi_j\text{ is $Y$-periodic},\qquad \int_{Y} \chi_j = 0.
\end{align}
Now, for $z\in L^2_{\mathrm{per},0}(Y)$, we can rewrite the dual problem
\begin{align}\label{dual with r}
-D^2:(rA\psi) + \nabla\cdot (rb\psi) = z\quad\text{in } Y,\qquad \psi\text{ is $Y$-periodic},\qquad \int_{Y} \psi = 0,
\end{align}
using that $\nabla\cdot \beta = 0$ weakly in $Y$ by definition of $r$ (see Proposition \ref{Prop: r}), as
\begin{align}\label{dual with r again}
-\nabla\cdot(rA\nabla \psi) + \beta\cdot \nabla \psi = z\quad\text{in } Y,\qquad \psi\text{ is $Y$-periodic},\qquad \int_{Y} \psi = 0.
\end{align}
Since $\beta \in L^{p}_{\mathrm{per}}(Y;\R^n)$ with $p>n$, we have that $\psi\in H^2_{\mathrm{per}}(Y)$ and $\|\psi\|_{H^2(Y)}\leq C \|z\|_{L^2(Y)}$ for some constant $C>0$. 

Noting that the underlying variational formulation of \eqref{chij with r} still obeys a G{\aa}rding inequality as well as boundedness, a duality argument similar to the one given in the proof of Theorem \ref{Thm: Approx r in set A} leads to an error bound for the approximation of $\chi_j$ based on \eqref{chij with r} in the case when $r$ is known.

Now, in general, $r$ is not known and we need to incorporate into our numerical method a finite element method for the approximation of $r$ from Section \ref{Sec: FEM r Set A}. We choose piecewise affine finite elements for simplicity. To this end, consider a shape-regular triangulation $\calT_h$ of $\overline{Y}$ into simplices with maximal overall edge-length $h>0$ that is consistent with the requirement of periodicity.

\begin{theorem}[Approximation of $\chi_j$ in setting $\calA$]\label{Thm: chij app}
Let $(A,b)\in \calA$, $j\in \{1,\dots,n\}$, and let $\chi_j\in H^1_{\mathrm{per},0}(Y)$ denote the unique solution to \eqref{chij pro}. Let $R_h$ denote the finite-dimensional subspace of $H^1_{\mathrm{per},0}(Y)$ consisting of continuous $Y$-periodic piecewise affine functions on $\calT_h$ with zero mean over $Y$. Let $(r_h)_{h>0}\subset W^{1,p}_{\mathrm{per}}(Y)$ with $p>n$ as in \eqref{AinW1p} be such that $e_h := \|r-r_h\|_{W^{1,p}(Y)}\rightarrow 0$ as $h\rightarrow 0$, where $r\in H^1_{\mathrm{per}}(Y)$ denotes the unique solution to \eqref{r problem}. Then, for $h>0$ sufficiently small, there exists a unique $\chi_{j,h}\in R_h$ such that 
\begin{align}\label{chijh pro}
\int_Y r_hA\nabla \chi_{j,h}\cdot \nabla v_h - \int_Y v_h (r_h b-\mathrm{div}(r_h A))\cdot \nabla \chi_{j,h} = \int_Y r_h b_j v_h\quad \forall v_h\in  R_h.
\end{align}
Furthermore, the following error bounds hold:
\begin{align}\label{errorbd chijh H1}
\|\chi_j - \chi_{j,h}\|_{H^1(Y)} \lesssim h \|\chi_j\|_{H^2(Y)} + \sqrt{e_h}(1+\|\chi_j\|_{H^2(Y)}),
\end{align}
and
\begin{align}\label{errorbd chijh L2}
\|\chi_j - \chi_{j,h}\|_{L^2(Y)} &\lesssim (h+e_h)\|\chi_j - \chi_{j,h}\|_{H^1(Y)}  + e_h(1+\|\chi_j\|_{H^1(Y)}),
\end{align}
where the constant absorbed in $``\lesssim"$ only depends on $(A,b)$ and $n$.
\end{theorem}

\begin{proof}
We set $\beta_h:=r_h b-\mathrm{div}(r_h A)$ and $\beta:=r b-\mathrm{div}(r A)$. We introduce
\begin{align*}
a_h(q,v)&:=\int_Y r_hA\nabla q\cdot \nabla v - \int_Y v \beta_h\cdot \nabla q,\qquad l_h(v):=\int_Y r_h b_j v,\\
a_0(q,v)&:= \int_Y r A\nabla q\cdot \nabla v - \int_Y v \beta \cdot \nabla q,\qquad\quad l_0(v):=\int_Y r b_j v
\end{align*}
for $q,v\in H^1_{\mathrm{per},0}(Y)$. Then, \eqref{chijh pro} reads $a_h(\chi_{j,h},v_h) = l_h(v_h)$ for all $v_h\in R_h$, and we have from \eqref{chij with r} that $a_0(\chi_{j},v) = l_0(v)$ for all $v\in H^1_{\mathrm{per},0}(Y)$. Using the assumed uniform ellipticity of $A$, positivity of $r$, and the fact that $\beta\in L^p_{\mathrm{per}}(Y;\R^n)$ with $p>n$, it is easily seen that we have the G{\aa}rding inequality
\begin{align*}
a_0(v,v) \geq \frac{\lambda \inf_{\R^n} r}{2} \|v\|_{H^1(Y)}^2 - c_1 \|v\|_{L^2(Y)}^2\qquad \forall v\in H^1_{\mathrm{per},0}(Y)
\end{align*}
for some constant $c_1 > 0$. Using that $\|r_h-r\|_{L^\infty(Y)} + \|\beta_h-\beta\|_{L^p(Y;\R^n)}\lesssim e_h$ and the Sobolev embedding $H^1(Y)\hookrightarrow L^{\frac{2p}{p-2}}(Y)$ as $p>n$, we also have that
\begin{align*}
\lvert a_h(q,v)-a_0(q,v)\rvert \lesssim e_h\|q\|_{H^1(Y)}\|v\|_{H^1(Y)}\qquad \forall q,v\in H^1_{\mathrm{per},0}(Y).
\end{align*}

\textit{Uniqueness of $\chi_{j,h}$:} Suppose that $\chi_{j,h}^{(1)},\chi_{j,h}^{(2)}\in R_h$ are two solutions to \eqref{chijh pro}, and set $z_{h}:=\chi_{j,h}^{(1)} - \chi_{j,h}^{(2)}$. Noting that since $z_h\in R_h\subset H^1_{\mathrm{per},0}(Y)$ there holds $a_h(z_h,v_h) = 0$ for all $v_h\in R_h$, we find that
\begin{align*}
\|z_h\|_{H^1(Y)}^2 &\lesssim a_0(z_h,z_h) + \|z_h\|_{L^2(Y)}^2 \\ &\lesssim  a_h(z_h,z_h) + \|z_h\|_{L^2(Y)}^2 + e_h\|z_h\|_{H^1(Y)}^2  \lesssim \|z_h\|_{L^2(Y)}^2 + e_h\|z_h\|_{H^1(Y)}^2.
\end{align*} 
In particular, as $e_h\rightarrow 0$, we have for $h>0$ sufficiently small that
\begin{align}\label{H1bddbyL2}
\|z_h\|_{H^1(Y)} \lesssim \|z_h\|_{L^2(Y)}.
\end{align}
Let $\psi\in H^1_{\mathrm{per},0}(Y)$ denote the unique solution to the dual problem \eqref{dual with r} with $z = z_h$, or equivalently, \eqref{dual with r again} with $z = z_h$ (existence and uniqueness of $\psi$ follow from Proposition \ref{Prop: FPK inhom A}, noting $(rA,rb)\in \calA$ and $z_h\in L^2_{\mathrm{per},0}(Y)$). Then, in view of the discussion following \eqref{dual with r}, we have that $\psi\in H^2_{\mathrm{per},0}(Y)$ and $\|\psi\|_{H^2(Y)}\lesssim \|z_h\|_{L^2(Y)}$. Denoting the piecewise linear quasi-interpolant of $\psi$ by  $\mathcal{I}_h \psi$, we introduce $\psi_{\mathcal{I}}:=\mathcal{I}_h \psi - \int_Y \mathcal{I}_h \psi = \mathcal{I}_h \psi + \int_Y (\psi-\mathcal{I}_h \psi) \in R_h$. Then, using that $\|\psi-\psi_{\mathcal{I}}\|_{H^1(Y)}\lesssim h \|\psi\|_{H^2(Y)} \lesssim h\|z_h\|_{L^2(Y)}$, we obtain that (recall that $\nabla\cdot \beta = 0$ weakly)
\begin{align*}
\|z_h\|_{L^2(Y)}^2 = a_0(z_h,\psi)  &\lesssim a_h(z_h,\psi) + e_h\|z_h\|_{H^1(Y)}\|\psi\|_{H^1(Y)} \\
&\lesssim a_h(z_h, \psi-\psi_{\mathcal{I}}) + e_h\|z_h\|_{H^1(Y)}\|z_h\|_{L^2(Y)}\\
&\lesssim a_0(z_h, \psi-\psi_{\mathcal{I}}) + e_h\|z_h\|_{H^1(Y)} \|z_h\|_{L^2(Y)} \\
&\lesssim \|z_h\|_{H^1(Y)} \|\psi-\psi_{\mathcal{I}}\|_{H^1(Y)} + e_h\|z_h\|_{H^1(Y)} \|z_h\|_{L^2(Y)} \\
&\lesssim (h +  e_h) \|z_h\|_{H^1(Y)} \|z_h\|_{L^2(Y)}.
\end{align*}
Combining this with \eqref{H1bddbyL2} yields $z_h = 0$ for $h>0$ sufficiently small, i.e., there is at most one solution to \eqref{chijh pro}.

\medskip
\textit{Existence of $\chi_{j,h}$:} As $R_h$ is finite-dimensional, uniqueness implies existence of a solution $\chi_{j,h}\in R_h$ to \eqref{chijh pro}.  

\medskip
\textit{Error bound:} Let $\Psi\in H^1_{\mathrm{per},0}(Y)$ denote the unique solution to the dual problem \eqref{dual with r} (or equivalently \eqref{dual with r again}) with $z := \chi_j-\chi_{j,h}\in H^1_{\mathrm{per},0}(Y)$. Note that $\Psi\in H^2_{\mathrm{per},0}(Y)$ and $\|\Psi\|_{H^2(Y)}\lesssim \|z\|_{L^2(Y)}$. Writing $\Psi_{\mathcal{I}}:=\mathcal{I}_h \Psi - \int_Y \mathcal{I}_h \Psi\in R_h$, we obtain
\begin{align*}
&\|z\|_{L^2(Y)}^2 = a_0(z,\Psi)
\\ &\lesssim a_h(z,\Psi) + e_h\|z\|_{H^1(Y)}\|\Psi\|_{H^1(Y)} \\
&\lesssim a_h(z,\Psi-\Psi_{\mathcal{I}}) + [l_0-l_h](\Psi_{\mathcal{I}}) + [a_h-a_0](\chi_j,\Psi_{\mathcal{I}})+  e_h\|z\|_{H^1(Y)}\|z\|_{L^2(Y)} \\
&\lesssim a_0(z,\Psi-\Psi_{\mathcal{I}}) + [l_0-l_h](\Psi_{\mathcal{I}}) + [a_h-a_0](\chi_j,\Psi_{\mathcal{I}}) +  e_h\|z\|_{H^1(Y)}\|z\|_{L^2(Y)} \\
&\lesssim (h+e_h)\|z\|_{H^1(Y)}\|z\|_{L^2(Y)} + [l_0-l_h](\Psi_{\mathcal{I}}) + [a_h-a_0](\chi_j,\Psi_{\mathcal{I}}) 
\\ &\lesssim (h+e_h)\|z\|_{H^1(Y)}\|z\|_{L^2(Y)}  + e_h\|z\|_{L^2(Y)} + e_h \|\chi_j\|_{H^1(Y)}\|z\|_{L^2(Y)}
\end{align*}
for $h>0$ sufficiently small, which implies \eqref{errorbd chijh L2}. Writing $\chi_{j,\mathcal{I}}:= \mathcal{I}_h \chi_j - \int_Y \mathcal{I}_h \chi_j \in R_h$ and $\eta_h:=\chi_{j,\mathcal{I}} - \chi_{j,h} = z -  (\chi_j - \chi_{j,\mathcal{I}})\in R_h$, we have that
\begin{align*}
&\|z\|_{H^1(Y)}^2 \lesssim a_0(z,z) + \|z\|_{L^2}^2 \\
&\lesssim a_0(z,\chi_j-\chi_{j,\mathcal{I}}) + \|z\|_{L^2}^2 +   a_0(z,\eta_h)\\
&\lesssim h\|\chi_j\|_{H^2}\|z\|_{H^1} + \|z\|_{L^2}^2 + a_h(z,\eta_h) +e_h\|\eta_h\|_{H^1} \|z\|_{H^1} \\
&\lesssim h\|\chi_j\|_{H^2}\|z\|_{H^1} + \|z\|_{L^2}^2 + [l_0-l_h](\eta_h) + [a_h-a_0](\chi_j,\eta_h)  +e_h\|\eta_h\|_{H^1} \|z\|_{H^1} \\
&\lesssim h\|\chi_j\|_{H^2}\|z\|_{H^1} + \|z\|_{L^2}^2  + e_h\|\eta_h\|_{H^1} (1+\|\chi_j\|_{H^1}+\|z\|_{H^1}) \\ 
&\lesssim h\|\chi_j\|_{H^2}\|z\|_{H^1} + \|z\|_{L^2}^2  + e_h(h\|\chi_j\|_{H^2} + \|z\|_{H^1}) (1+\|\chi_j\|_{H^1}+\|z\|_{H^1}) \\
&\lesssim h\|\chi_j\|_{H^2}\|z\|_{H^1} + \|z\|_{L^2}^2  + e_h(1+\|\chi_j\|_{H^2}^2 + \|z\|_{H^1}^2),
\end{align*}
where we wrote $\|\cdot\|_{L^2}=\|\cdot\|_{L^2(Y)}$ and $\|\cdot\|_{H^1}=\|\cdot\|_{H^1(Y)}$. Using \eqref{errorbd chijh L2}, we obtain
\begin{align*}
\|z\|_{H^1(Y)}^2 &\lesssim h^2 \|\chi_j\|_{H^2(Y)}^2 + \|z\|_{L^2(Y)}^2  + e_h(1+\|\chi_j\|_{H^2(Y)}^2) \\
&\lesssim h^2 \|\chi_j\|_{H^2(Y)}^2 + (h+e_h)^2\|z\|_{H^1(Y)}^2  + e_h(1+\|\chi_j\|_{H^2(Y)}^2)
\end{align*}
for $h>0$ sufficiently small. Absorbing the second term on the right-hand side into the left-hand side, we conclude that \eqref{errorbd chijh H1} holds.
\end{proof}

We are now in a situation to state a result concerning the approximation of the effective diffusion matrix.

\begin{corollary}[Approximation of $\overline{A}$ in setting $\calA$]\label{Cor: Abar A}

Suppose that the assumptions of Theorem \ref{Thm: chij app} hold, and let $\overline{A}$ be given by \eqref{def Abar}. Then, by introducing 
\begin{align*}
\overline{A}_h := \int_Y r_h[I_n + D\chi_h]A[I_n + (D\chi_h)^{\mathrm{T}}],\qquad \chi_h := (\chi_{1,h},\dots,\chi_{n,h}), 
\end{align*}
we have for $h>0$ sufficiently small that
\begin{align*}
\left\lvert \overline{A} - \overline{A}_h \right\rvert \lesssim \left(h \|\chi\|_{H^2(Y;\R^n)} + \sqrt{e_h}\left(1+\|\chi\|_{H^2(Y;\R^n)}\right)\right)\left(1+\|D\chi\|_{L^2(Y;\R^{n\times n})}\right), 
\end{align*}
where the constant absorbed in $``\lesssim"$ only depends on $(A,b)$ and $n$.
\end{corollary} 

\begin{proof}

Writing $d_h:= \|D(\chi-\chi_h)\|_{L^2(Y;\R^{n\times n})}$, we have by the triangle inequality that $\left\lvert \overline{A} - \overline{A}_h \right\rvert \leq T_1 + T_2 + T_3$, where
\begin{align*}
T_1 &:= \left\lvert \left\langle\left[I_n+D\chi \right]A \left(D(\chi-\chi_h)\right)^{\mathrm{T}}\right\rangle \right\rvert \lesssim  d_h\left(1+\|D\chi\|_{L^2(Y;\R^{n\times n})}\right),\\
T_2 &:= \left\lvert \left\langle \left(D(\chi-\chi_h)\right)A \left[I_n + (D\chi_h)^{\mathrm{T}}\right]\right\rangle \right\rvert \lesssim d_h\left(1+\|D\chi_h\|_{L^2(Y;\R^{n\times n})}\right),\\
T_3 &:= \left\lvert \int_Y (r-r_h)\left[I_n+D\chi_h\right] A \left[I_n + (D\chi_h)^{\mathrm{T}}\right]\right\rvert \lesssim e_h \left(1+\|D\chi_h\|_{L^2(Y;\R^{n\times n})}\right)^2.
\end{align*}
We have used that $A\in L^{\infty}(Y;\R^{n\times n})$, $r\in L^\infty(Y)$, and $\|r-r_h\|_{L^{\infty}(Y)}\lesssim e_h$ by Sobolev embedding. By \eqref{errorbd chijh H1}, we deduce for $h>0$ sufficiently small that
\begin{align*}
\left\lvert \overline{A} - \overline{A}_h \right\rvert &\lesssim  \left(d_h + e_h(1+\|D\chi\|_{L^2(Y;\R^{n\times n})})\right)\left(1+\|D\chi\|_{L^2(Y;\R^{n\times n})}\right) \\
&\lesssim \left(h \|\chi\|_{H^2(Y;\R^{n})} + \sqrt{e_h}\left(1+\|\chi\|_{H^2(Y;\R^{n})}\right)\right)\left(1+\|D\chi\|_{L^2(Y;\R^{n\times n})}\right), 
\end{align*}
as required. 
\end{proof}

\subsection{Setting $\calB$}\label{Sec: SecAbar B}

We now discuss the case $(A,b)\in \calB$, and we write $(\tilde{A},\tilde{b})$ to denote the pair of renormalized coefficients defined by \eqref{gam}, \eqref{At Bt}. We denote the unique solution to the periodic FPK problem \eqref{rt prob} by $\tilde{r}$ (given by Theorem \ref{Thm: ana in set B}). Recall from Corollary \ref{Cor: r in setB} that the invariant measure $r$, i.e., the unique solution to \eqref{r prob setB} in $L^2_{\mathrm{per}}(Y)$, is given by
\begin{align}\label{r in terms rt}
r = \frac{1}{\int_Y \gamma \tilde{r}}\gamma \tilde{r}.
\end{align}
We assume the centering condition
\begin{align}\label{cent ag}
\langle b\rangle = 0,
\end{align}
where we use the notation $\langle \fhi \rangle := \int_Y \fhi r$. Let us discuss the well-posedness and finite element approximation of the problem \eqref{chij pro} in this setting. Using techniques similar to Section \ref{Sec: setting B}, we can show the existence and uniqueness of a solution $\chi_j\in H^2_{\mathrm{per},0}(Y)$ to \eqref{chij pro}, as well as construct a simple finite element scheme for its approximation.

\begin{theorem}[Well-posedness and finite element approximation of \eqref{chij pro}]\label{Thm: chij app set B}

Let $(A,b)\in \calB$. Let $(\tilde{A},\tilde{b})$ denote the pair of renormalized coefficients given by \eqref{gam}, \eqref{At Bt}, and suppose that \eqref{cent ag} holds. Further, let $B_1:H^2_{\mathrm{per},0}(Y)\times H^2_{\mathrm{per},0}(Y)\rightarrow \R$ denote the bilinear form defined in \eqref{B1 defn}, and let $B_2:H^1_{\mathrm{per},0}(Y;\R^n)\times H^1_{\mathrm{per},0}(Y;\R^n)\rightarrow \R$ denote the bilinear form from Lemma \ref{Lmm: char of rt}. Then, the following assertions hold.
\begin{itemize}
\item[(i)] There exists a unique $\chi_j\in H^2_{\mathrm{per},0}(Y)$ such that $B_1(\fhi,\chi_j) = -\int_Y \tilde{b}_j\Delta \fhi$ for all $\fhi\in H^2_{\mathrm{per},0}(Y)$. Further, $\chi_j$ is the unique solution to \eqref{chij pro} in $H^2_{\mathrm{per},0}(Y)$.
\item[(ii)] The function $\xi_j:=\nabla \chi_j$ is the unique element $\xi_j\in H^1_{\mathrm{per},0}(Y;\R^n)$ that satisfies $B_2(w,\xi_j) = -\int_Y \tilde{b}_j \nabla\cdot w$ for all $w\in H^1_{\mathrm{per},0}(Y;\R^n)$. Further, for any closed linear subspace $P_h$ of $H^1_{\mathrm{per},0}(Y;\R^n)$, there exists a unique $\xi_{j,h}\in P_h$ such that $B_2(w_h,\xi_{j,h}) = -\int_Y \tilde{b}_j \nabla\cdot w_h$ for all $w_h\in P_h$, and we have the bound 
\begin{align*}
\|\nabla \chi_j - \xi_{j,h}\|_{H^1(Y;\R^n)} \leq C \inf_{w_h\in P_h} \|D(\nabla \chi_j - w_h)\|_{L^2(Y;\R^{n\times n})}
\end{align*}
for some constant $C=C(\lambda,\Lambda,\delta,n)>0$.
\end{itemize}

\end{theorem}

\begin{proof}
(i) By the proof of Theorem \ref{Thm: Approx r in set A}, we know that $B_1$ is a coercive bounded bilinear form on $H^2_{\mathrm{per},0}(Y)$. Hence, the first part of (i) follows from the Lax--Milgram theorem. For the second part of (i), it is clear that if $\chi_j\in H^2_{\mathrm{per},0}(Y)$ solves \eqref{chij pro}, then $B_1(\fhi,\chi_j) = \int_Y \tilde{b}_j(-\Delta \fhi)$ for all $\fhi\in H^2_{\mathrm{per},0}(Y)$. For the converse, suppose $\chi_j\in H^2_{\mathrm{per},0}(Y)$ satisfies $B_1(\fhi,\chi_j) = \int_Y \tilde{b}_j(-\Delta \fhi)$ for all $\fhi\in H^2_{\mathrm{per},0}(Y)$, and let $\Phi\in L^2_{\mathrm{per}}(Y)$. Then, there exists a unique $\fhi\in H^2_{\mathrm{per},0}(Y)$ such that $\Phi = c \tilde{r}-\Delta \fhi $ with $c:= \int_Y \Phi$. In view of \eqref{rt prob} and \eqref{cent ag}, it follows that  
\begin{align*}
\int_Y \left(-\tilde{A}:D^2 \chi_j - \tilde{b}\cdot \nabla \chi_j \right)\Phi  = B_1(\fhi,\chi_j) = \int_Y \tilde{b}_j(-\Delta \fhi) = \int_Y \tilde{b}_j\Phi,
\end{align*}
where we used that $\int_Y \tilde{b}_j \tilde{r}  = \langle b_j\rangle\int_Y \gamma \tilde{r} = 0$. Since $\Phi\in L^2_{\mathrm{per}}(Y)$ was arbitrary and $\gamma > 0$ a.e., we obtain that $\chi_j$ solves \eqref{chij pro}.  

(ii) By (i) and the definition of $B_2$, we have that $\xi_j:=\nabla \chi_j\in H^1_{\mathrm{per},0}(Y;\R^n)$ and $B_2(w,\xi_j) = \int_Y \tilde{b}_j (-\nabla\cdot w)$ for all $w\in H^1_{\mathrm{per},0}(Y;\R^n)$. The uniqueness of $\xi_j$ follows from coercivity of $B_2$ on $H^1_{\mathrm{per},0}(Y;\R^n)$. The existence and uniqueness of $\xi_{j,h}\in P_h$ follows from the Lax--Milgram theorem and the error bound follows from a standard Galerkin orthogonality argument.
\end{proof}

Let us give some comments regarding the homogenization of \eqref{ueps prob} in this weak regularity setting $(A,b)\in \calB$, assuming that $\tilde{r}\in L^\infty_{\mathrm{per}}(Y;(0,\infty))$. Then, scaling equation \eqref{ueps prob} by the invariant measure $r$ from \eqref{r in terms rt} and applying the transformation argument from \cite{AL89,JZ23} yields the problem
\begin{align}\label{Qprob}
\begin{split}
-\nabla\cdot \left(Q^\eps \nabla u_\eps\right)  &= f\quad\text{in }\Omega,\\  u_{\eps} &= g\quad\text{on }\partial\Omega,
\end{split}
\end{align}
where $Q^\eps := Q(\frac{\cdot}{\eps})$ with $Q:= rA + \Psi$, where $\Psi = (\psi_{ij})_{1\leq i,j\leq n}$ with $\psi_{ij} := \partial_i \phi_j - \partial_j \phi_i$ and 
\begin{align*}
-\Delta \phi_j = \nabla\cdot (rAe_j)-rb_j\quad\text{in }Y,\qquad \phi_j\text{ is $Y$-periodic},\qquad \int_{Y} \phi_j = 0.
\end{align*}
If $\phi_j\in W^{1,\infty}_{\mathrm{per}}(Y)$, then for any sufficiently regular $f,g,\partial\Omega$, we have that $Q\in L^\infty_{\mathrm{per}}(Y;\R^{n\times n})$ is uniformly elliptic and there exists a unique solution $u_\eps \in H^2(\Omega)$ to \eqref{Qprob}, which is equivalent to \eqref{ueps prob}. As $\eps\rightarrow 0$, the solution $u_{\eps}$ converges weakly in $H^1(\Omega)$ to the solution $u$ of the homogenized problem \eqref{u pro} with $\overline{A}$ as in \eqref{def Abar}.

\begin{corollary}[Approximation of $\overline{A}$ in setting $\calB$]\label{Cor: Abar B}

Suppose that the assumptions of Theorem \ref{Thm: chij app set B} hold, and let $\xi_h$ denote a finite element approximation of $D\chi$ obtained by Theorem \ref{Thm: chij app set B}(ii) with $\|D\chi - \xi_h\|_{H^1(Y;\R^{n\times n})}\rightarrow 0$, where $\chi = (\chi_1,\dots,\chi_n)$. Let $\tilde{r}$ be given by Lemma \ref{Lmm: char of rt}, and let $\tilde{r}_h$ denote its finite element approximation obtained by Theorem \ref{Thm: FEM rtilde} with $\|\tilde{r}-\tilde{r}_h\|_{L^2(Y)}\rightarrow 0$. Suppose that $n\leq 4$ and let $\overline{A}$ be given by \eqref{def Abar} with $r$ given by \eqref{r in terms rt}. Then, for $h>0$ sufficiently small, by introducing 
\begin{align*}
\overline{A}_h := \int_Y r_h[I_n + \xi_h]A[I_n + \xi_h^{\mathrm{T}}],\qquad r_h:=\frac{1}{\int_Y \gamma \tilde{r}_h}\gamma \tilde{r}_h,
\end{align*}
we have the bound 
\begin{align}\label{Abarcl}
\begin{split}
\left\lvert \overline{A} - \overline{A}_h \right\rvert &\lesssim \|D\chi-\xi_h\|_{H^1(Y;\R^{n\times n})}\left(1+\|D\chi\|_{H^1(Y;\R^{n\times n})}\right) \\ &\qquad+ \|\tilde{r}-\tilde{r}_h\|_{L^2(Y)}\left(1+\|D\chi\|_{H^1(Y;\R^{n\times n})}\right)^2,
\end{split}
\end{align}
where the constant absorbed in $``\lesssim"$ only depends on $(A,b)$ and $n$.
\end{corollary} 

\begin{proof}
We begin by noting that $H^1(Y)\hookrightarrow L^4(Y)$ for $n\leq 4$ by Sobolev embedding, and that $\|r-r_h\|_{L^2(Y)}\lesssim \|\tilde{r}-\tilde{r}_h\|_{L^2(Y)}$. By the triangle inequality, we have that $\left\lvert \overline{A} - \overline{A}_h \right\rvert \leq T_1 + T_2 + T_3$, where
\begin{align*}
T_1 &:= \left\lvert \left\langle\left(I_n+D\chi \right)A \left(D\chi-\xi_h\right)^{\mathrm{T}}\right\rangle \right\rvert \lesssim  \|D\chi-\xi_h\|_{H^1(Y;\R^{n\times n})}\left(1+\|D\chi\|_{H^1(Y;\R^{n\times n})}\right),\\
T_2 &:= \left\lvert \left\langle \left(D\chi-\xi_h\right)A \left(I_n + \xi_h^{\mathrm{T}}\right)\right\rangle \right\rvert \lesssim \|D\chi-\xi_h\|_{H^1(Y;\R^{n\times n})}\left(1+\|\xi_h\|_{H^1(Y;\R^{n\times n})}\right),\\
T_3 &:= \left\lvert \int_Y (r-r_h)\left(I_n+\xi_h\right) A \left(I_n + \xi_h^{\mathrm{T}}\right)\right\rvert \lesssim \|\tilde{r}-\tilde{r}_h\|_{L^2(Y)} \left(1+\|\xi_h\|_{H^1(Y;\R^{n\times n})}\right)^2,
\end{align*}
and it follows that \eqref{Abarcl} holds for $h>0$ sufficiently small.
\end{proof}

\section{Numerical Experiments}\label{Sec: NumExp}

We provide one numerical experiment for the setting $(A,b)\in \mathcal{A}$, and one numerical experiment for the setting $(A,b)\in \calB$. Both experiments are for dimension $n = 2$.

\subsection{Setting $\calA$}

We choose $A:\R^2\rightarrow \R^{2\times 2}$ and $b:\R^2\rightarrow \R^2$ to be 
\begin{align}\label{A,b choice exp set A}
\begin{split}
A(y) &:= \begin{pmatrix}
1+\arcsin(\sin^2(\pi y_1)) & \frac{1}{2}\sin(2\pi y_1) \\
\frac{1}{2}\sin(2\pi y_1) & 2+\cos^2(\pi y_1)
\end{pmatrix},\\ b(y)&:=\mathrm{sign}(\sin(2\pi y_1))\begin{pmatrix}
1\\1
\end{pmatrix}
\end{split}
\end{align}
for $y=(y_1,y_2)\in \R^2$, where the choice of $A$ is as in \cite{CSS20}. We note that $(A,b)\in \calA$ since $A\in W^{1,\infty}_{\mathrm{per}}(Y;\R^{2\times 2}_{\mathrm{sym}})$, $b\in L^\infty_{\mathrm{per}}(Y;\R^2)$, and $A$ is uniformly elliptic. 

First, we test the finite element scheme from Theorems \ref{Thm: Approx r in set A} and \ref{Thm: Lp est r} for the approximation of the invariant measure, i.e., the unique solution $r$ to the FPK problem \eqref{r problem}, where we choose $R_h$ to be the space consisting of $Y$-periodic continuous piecewise affine functions with zero mean over $Y$ on a periodic shape-regular triangulation $\calT_h$ of the unit cell into triangles with vertices $(ih,jh)$, $1\leq i,j\leq N = \frac{1}{h}\in \N$. All computations are performed in FreeFem\texttt{++}; see \cite{Hec12}. By introducing
\begin{align*}
K:\R\rightarrow \R,\qquad K(t):= \int_0^{t}\frac{\mathrm{sign}(\sin(2\pi x))}{1+\arcsin(\sin^2(\pi x))}\mathrm{d}x,
\end{align*}
we compare our approximation with the true solution given by
\begin{align*}
r(y)= C_1^{-1}\frac{\mathrm{e}^{K(y_1)}}{1+\arcsin(\sin^2(\pi y_1))},\qquad C_1:= \int_0^1 \frac{\mathrm{e}^{K(t)}}{1+\arcsin(\sin^2(\pi t))}\mathrm{d}t
\end{align*}
for $y=(y_1,y_2)\in\R^2$; see \cite{JZ23}. 

The approximation errors are shown in Figure \ref{Fig: err curves Set A}(A). For $p\in \{2,3\}$, we observe convergence of order $\calO(h^{\frac{1}{p}})$ in the $W^{1,p}(Y)$-norm and convergence of order $\calO(h^{1+\frac{1}{p}})$ in the $L^p(Y)$-norm, which is consistent with the bounds from Theorems \ref{Thm: Approx r in set A} and \ref{Thm: Lp est r} since $r\in W^{1+s,p}(Y)$ for any $s\in [0,\frac{1}{p})$. Further, we observe superconvergence of order $\calO(h)$ in the $W^{1,p}(Y)$-norm and of order $\calO(h^2)$ in the $L^p(Y)$-norm when there are no elements of $\calT_h$ whose interior intersects the line $\{y_1 = \frac{1}{2}\}$ along which $\partial_1 r$ jumps, which is expected since $\left. r\right\rvert_{Q\times (0,1)}\in H^2(Q\times (0,1))$ for $Q\in \{(0,\frac{1}{2}),(\frac{1}{2},1)\}$.

Since $K(1) = 0$, the centering condition \eqref{cent cond} is satisfied; see \cite{JZ23}.
We now test the finite element scheme from Theorem \ref{Thm: chij app} for the approximation of \eqref{chij pro}. We compare with the true solution given by
\begin{align*}
\chi_j(y) = C_2^{-1}\int_0^{y_1} \mathrm{e}^{-K(t)}\mathrm{d}t-y_1 + c,\qquad C_2:=\int_0^1 \mathrm{e}^{-K(t)}\mathrm{d}t
\end{align*}
for $y=(y_1,y_2)\in \R^2$ and $j\in \{1,2\}$, where $c$ is a constant such that $\int_Y \chi_j = 0$. 

Finally, we test the approximation of the effective diffusion matrix from Corollary \ref{Cor: Abar A}. We compare with the true effective diffusion matrix $\overline{A}\in \R^{2\times 2}_{\mathrm{sym}}$ given by \eqref{def Abar}. The approximation errors are shown in Figure \ref{Fig: err curves Set A}(B). For the approximation of $\chi_j$ we observe convergence of order $\calO(h)$ in the $H^1(Y)$-seminorm, and for the approximation of $\overline{A}$ we observe convergence of order $\calO(h^2)$ in the Frobenius-norm. The results are consistent with, and better than the  behavior expected from the bound in Corollary \ref{Cor: Abar A}.

\begin{figure}
	\begin{subfigure}{0.49\textwidth}
		\includegraphics[width=\textwidth]{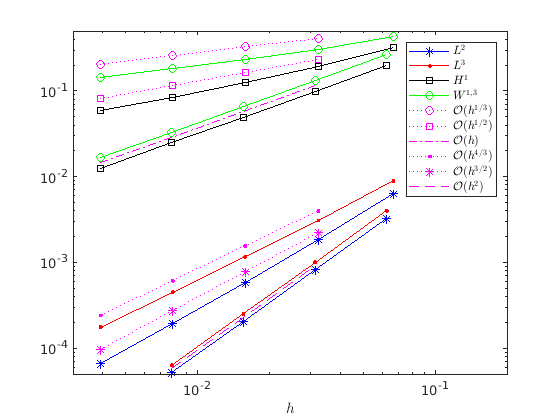}
		\subcaption{$\|r-r_h\|_{X(Y)}$ for $X\in\{ L^2,L^3,H^1,W^{1,3}\}$.}
	\end{subfigure}
	\begin{subfigure}{0.49\textwidth}
		\includegraphics[width=\textwidth]{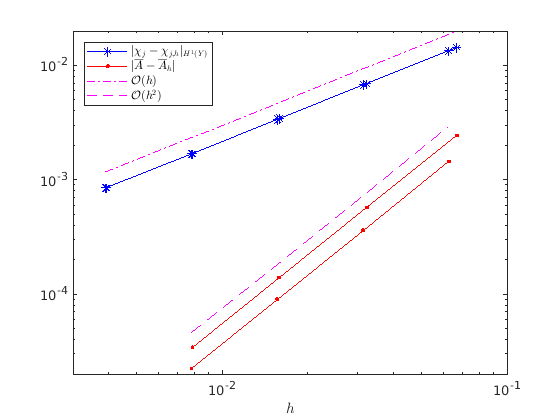}
		\subcaption{$\lvert \chi_j-\chi_{j,h}\rvert_{H^1(Y)}$ and $\lvert \overline{A}-\overline{A}_h\rvert$.}
	\end{subfigure}
	\caption{Approximation errors for the approximation of the invariant measure $r$ and the effective diffusion matrix $\overline{A}$ corresponding to $(A,b)\in \calA$ defined in \eqref{A,b choice exp set A}. We observe two curves, corresponding to whether or not there are elements of the triangulation whose interior
intersects the line $\{y_1 = \frac{1}{2}\}$ along which $\partial_1 r$ exhibits a jump.}
	\label{Fig: err curves Set A}
\end{figure}

\subsection{Setting $\calB$}

We choose $A:\R^2\rightarrow \R^{2\times 2}$ and $b:\R^2\rightarrow \R^2$ to be 
\begin{align}\label{A,b choice exp set B}
\begin{split}
A(y) &:= \begin{pmatrix}
2+\mathrm{sign}(\cos(\pi y_1))\sin(\pi y_1) & \frac{1}{2}\sin(2\pi y_1) \\
\frac{1}{2}\sin(2\pi y_1) & 2+\cos^2(\pi y_1)
\end{pmatrix},\\ b(y)&:= (\tfrac{1}{4}+\tfrac{3}{4}\mathrm{sign}(\sin(2\pi y_1)))\begin{pmatrix}
1\\1
\end{pmatrix}
\end{split}
\end{align}
for $y=(y_1,y_2)\in \R^2$. Note that $(A,b)\in \calB$ since $A\in L^{\infty}_{\mathrm{per}}(Y;\R^{2\times 2}_{\mathrm{sym}})$, $b\in L^\infty_{\mathrm{per}}(Y;\R^2)$, $A$ is uniformly elliptic, and the Cordes-type condition \eqref{mod Cordes} is satisfied  with, e.g., $\delta = \frac{1}{4} \in ( \frac{2}{2+\pi^2},1]$. 

First, we test the finite element scheme from Theorem \ref{Thm: FEM rtilde} in conjunction with \eqref{r in terms rt} for the approximation of the unique solution $r$ to the  FPK problem \eqref{r prob setB}, where we choose $P_h \subset H^1_{\mathrm{per},0}(Y;\R^2)$ to be the space consisting of vector-valued functions whose components are $Y$-periodic continuous piecewise affine functions with zero mean over $Y$ on a periodic shape-regular triangulation $\calT_h$ of the unit cell into triangles with vertices $(ih,jh)$, $1\leq i,j\leq N = \frac{1}{h}\in \N$. By introducing
\begin{align*}
\tilde{K}:\R\rightarrow \R,\qquad \tilde{K}(t):= \int_0^{t}\frac{\frac{1}{4}+\frac{3}{4}\mathrm{sign}(\sin(2\pi x))}{2+\mathrm{sign}(\cos(\pi x))\sin(\pi x)}\mathrm{d}x,
\end{align*}
we compare our approximation with the true solution given by
\begin{align*}
r(y)= \tilde{C}_1^{-1}\frac{\mathrm{e}^{\tilde{K}(y_1)}}{2+\mathrm{sign}(\cos(\pi y_1))\sin(\pi y_1)},\qquad \tilde{C}_1:= \int_0^1 \frac{\mathrm{e}^{\tilde{K}(t)}}{2+\mathrm{sign}(\cos(\pi t))\sin(\pi t)}\mathrm{d}t
\end{align*}
for $y=(y_1,y_2)\in\R^2$; see \cite{JZ23}. The approximation error in the $L^2(Y)$-norm is shown in Figure \ref{Fig: err curves Set B}(A). We observe convergence of order $\calO(\sqrt{h})$, and superconvergence of order $\calO(h)$ when there are no elements of $\calT_h$ whose interior intersects the line $\{y_1 = \frac{1}{2}\}$ along which $r$ jumps. This indicates that the function $\rho$ from Theorem \ref{Thm: FEM rtilde} belongs to $H^{1+s}(Y;\R^n)$ for any $s<\frac{1}{2}$, which is expected since $r\in H^s(Y)$ for any  $s<\frac{1}{2}$. Note also that $\left. r\right\rvert_{Q\times (0,1)}\in H^1(Q\times (0,1))$ for $Q\in \{(0,\frac{1}{2}),(\frac{1}{2},1)\}$.

Since $\tilde{K}(1) = 0$, the centering condition \eqref{cent ag} is satisfied; see \cite{JZ23}. We now test the finite element scheme from Theorem \ref{Thm: chij app set B} for the approximation of \eqref{chij pro}. We compare with the true solution given by
\begin{align*}
\chi_j(y) = \tilde{C}_2^{-1}\int_0^{y_1} \mathrm{e}^{-\tilde{K}(t)}\mathrm{d}t-y_1 + \tilde{c},\qquad \tilde{C}_2:=\int_0^1 \mathrm{e}^{-\tilde{K}(t)}\mathrm{d}t
\end{align*}
for $y=(y_1,y_2)\in \R^2$ and $j\in \{1,2\}$, where $\tilde{c}$ is a constant such that $\int_Y \chi_j = 0$. 

Finally, we test the approximation of the effective diffusion matrix from Corollary \ref{Cor: Abar B}. We compare with the true effective diffusion matrix $\overline{A}\in \R^{2\times 2}_{\mathrm{sym}}$ given by \eqref{def Abar}. The approximation errors are shown in Figure \ref{Fig: err curves Set B}(B). For the approximation of $\nabla\chi_j$ in the $H^1(Y;\R^2)$-norm we observe convergence of order $\calO(\sqrt{h})$, and superconvergence of order $\calO(h)$ when there are no elements of $\calT_h$ whose interior intersects the line $\{y_1 = \frac{1}{2}\}$. For the approximation of $\overline{A}$ in the Frobenius-norm, we observe convergence of order $\calO(h)$, and superconvergence of order $\calO(h^2)$ when there are no elements of $\calT_h$ whose interior intersects the line $\{y_1 = \frac{1}{2}\}$. The results are consistent with, and actually better than the expected behavior from the bound in Corollary \ref{Cor: Abar B}.

\begin{figure}
	\begin{subfigure}{0.49\textwidth}
		\includegraphics[width=\textwidth]{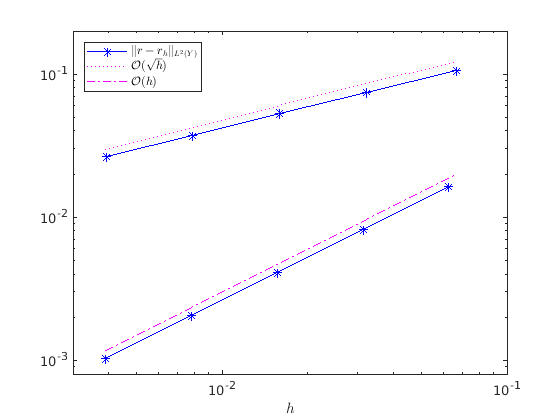}
		\subcaption{$\|r-r_h\|_{L^2(Y)}$.}
	\end{subfigure}
	\begin{subfigure}{0.49\textwidth}
		\includegraphics[width=\textwidth]{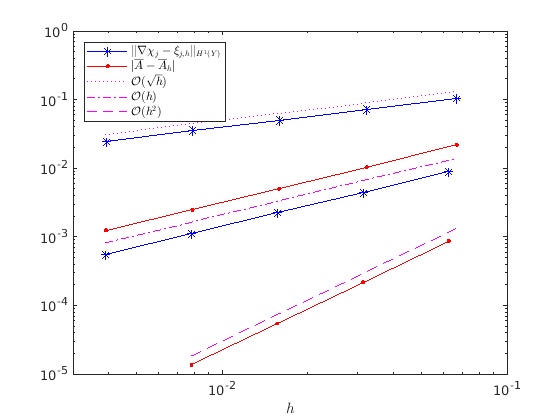}
		\subcaption{$\| \nabla \chi_j-\xi_{j,h}\|_{H^1(Y)}$ and $\lvert \overline{A}-\overline{A}_h\rvert$.}
	\end{subfigure}
	\caption{Approximation errors for the approximation of the invariant measure $r$ and the effective diffusion matrix $\overline{A}$ corresponding to $(A,b)\in \calB$ defined in \eqref{A,b choice exp set B}. We observe two curves, corresponding to whether or not there are elements of the triangulation whose interior
intersects the line $\{y_1 = \frac{1}{2}\}$ along which $r$ exhibits a jump.}
	\label{Fig: err curves Set B}
\end{figure}

\section*{Acknowledgments}

The work of ZZ is partially supported by the National Natural Science Foundation of China  (Project 12171406), the Hong Kong RGC grant (Projects 17307921 and 17304324), and Seed Funding for Basic Research at HKU.

\bibliographystyle{plain}
\bibliography{ref_nondivhom_large_drift_FEM}

\end{document}